\documentclass[12pt,oneside]{amsart}
\usepackage[numbers, sort, longnamesfirst]{natbib}
\usepackage{amsaddr}
\usepackage{amssymb,amsmath,mathtools}
\usepackage{enumerate}
\usepackage{xcolor}
\usepackage{mathrsfs}
\usepackage{enumitem}
\usepackage{graphicx, subcaption}
\usepackage{float}
\usepackage{dsfont}
\usepackage{bbm}

\usepackage{setspace}
\usepackage{geometry}
\usepackage[colorlinks,citecolor=blue,urlcolor=blue,linkcolor=blue, pdfborder={0 0 0}]{hyperref}
\usepackage{cleveref}
\usepackage{autonum}
\Crefname{ass}{Assumption \ref{ass:cpct}}{Assumptions \ref{ass:cpct}}

\newtheorem{theorem}{Theorem}[section]
\newtheorem{proposition}[theorem]{Proposition}
\newtheorem{lemma}[theorem]{Lemma}
\newtheorem{corollary}[theorem]{Corollary}
\theoremstyle{definition}

\newtheorem{remark}[theorem]{Remark}

\newtheorem{assumption}[theorem]{Assumption}

\newcommand{\F}{\mathcal F}

\newcommand{\Law}{\mathrm{Law}}

\newcommand{\id}{\mathrm{id}}

\newcommand{\D}{\mathrm{d}}

\newcommand{\W}{\mathcal W}
\newcommand{\N}{\mathbb N}
\newcommand{\R}{\mathbb R}
\newcommand{\X}{\mathcal X}

\newcommand{\Pc}{\mathcal P}
\newcommand{\Sc}{\mathcal S}
\newcommand{\Scc}{\hat{\mathcal S}}
\newcommand{\Sccc}{\bar{\mathcal S}}

\newcommand{\reg}{\mathrm{r}}

\newcommand{\trace}{\mathrm{Tr}}

\def\P{{\mathbb P}}
\def\E{{\mathbb E}}

\def\Q{{\mathbb Q}}
\def\R{{\mathbb R}}

\usepackage{relsize}
\def\fcmp{\mathbin{\raise 0.6ex\hbox{\oalign{\hfil$\scriptscriptstyle \mathrm{o}$\hfil\cr\hfil$\scriptscriptstyle\mathrm{9}$\hfil}}}}

\numberwithin{equation}{section}

\author{Gudmund Pammer}
\address[GP]{ETH Z\"urich, Z\"urich, Switzerland}
\email{gudmund.pammer@math.ethz.ch}
\author{Benjamin A.\ Robinson}
\address[BR, WS]{Universit\"at Wien, Vienna, Austria}
\email{ben.robinson@univie.ac.at}
\author{Walter Schachermayer}
\address[BR, WS]{Universit\"at Wien, Vienna, Austria}
\email{walter.schachermayer@univie.ac.at}

\thanks{This research was funded in part by the Austrian Science Fund (FWF) [10.55776/Y782], [10.55776/P34743], [10.55776/P35519], [10.55776/P35197]. For open access purposes, the author has applied a CC BY public copyright license to any author accepted manuscript version arising from this submission.}

\keywords{Kellerer's theorem, Strassen's theorem, mimicking martingales, peacocks, diffusion processes, martingale optimal transport}
\date{\today}
\subjclass[2020]{60G44, 60J60, 60H10, 60J25 (Primary) 60F99 (Secondary)}

\begin{document}
\title{A regularized Kellerer theorem in arbitrary dimension}

\begin{abstract}
We present a multidimensional extension of Kellerer's theorem on the existence of mimicking Markov martingales for peacocks, a term derived from the French for stochastic processes increasing in convex order. For a continuous-time peacock in arbitrary dimension, after Gaussian regularization, we show that there exists a strongly Markovian mimicking martingale It\^o diffusion. A novel compactness result for martingale diffusions is a key tool in our proof. Moreover, we provide counterexamples to show, in dimension $d \geq 2$, that uniqueness may not hold, and that some regularization is necessary to guarantee existence of a mimicking Markov martingale.
\end{abstract}

\maketitle

\section{Introduction}\label{sec:intro}

Given a finite set of probability measures on $\R^d$ that are increasing in convex order, Strassen \cite{St65} showed in 1965 that there exists a Markov martingale whose marginals coincide with the given probability measures. We call this latter property \emph{mimicking}. For a family of measures indexed by \emph{continuous time} that are increasing in convex order, also called a peacock (\emph{Processus Croissant pour l'Ordre Convexe}), Kellerer \cite{Ke72} proved in 1972 that there exists a mimicking strong Markov martingale in dimension one. The questions of continuity and uniqueness for Kellerer's mimicking martingale remained open until the work of Lowther \cite{Lo08b, Lo08a, Lo08d, Lo09, Lo08c} completely clarified the situation. Lowther showed that, in dimension one, there exists a \emph{unique strong Markov} mimicking martingale and, moreover this process has \emph{continuous paths} when the peacock is weakly continuous and the marginals have convex support. It is also known that the strong Markov property is required to obtain uniqueness; Beiglb\"ock, Lowther, Pammer and Schachermayer \cite{BePaSc21a} construct a one-dimensional continuous Markov martingale whose marginals coincide with those of Brownian motion but which does not have the strong Markov property.

While the problem of finding mimicking Markov martingales is thus very well understood for one-dimensional peacocks, the higher dimensional case has remained wide open, although 50 years have passed since the publication of Kellerer's result. In this paper, to the best of our knowledge, we provide the first known multidimensional extension of Kellerer's theorem. Given a peacock on $\R^d$, we show that, after some Gaussian regularization, there exists a \emph{strongly Markovian martingale It\^o diffusion} that mimics the regularized peacock.

To prove our result, we construct a martingale It\^o diffusion that mimics the regularized peacock on the dyadics, and then pass to a limit in finite dimensional distributions. In order to take such a limit, we prove a compactness result for martingale It\^o diffusions.

Additionally, we show that uniqueness does not necessarily hold in higher dimensions. We consider an example of a martingale It\^o diffusion studied by Robinson \cite{Ro20} and Cox--Robinson \cite{CoRo22}, and we show that this martingale mimics the marginals of a two-dimensional Brownian motion, while itself not being a Brownian motion. We also show, by means of counterexamples, that the Gaussian regularization is necessary to guarantee the existence of a mimicking Markov martingale.

For $\eta > 0$, let $\gamma^\eta$ denote the centered Gaussian law on $\R^d$ with covariance $\eta \id$, and let $\ast$ denote the convolution operator between measures. Our main result is the following.

\begin{theorem}[existence of mimicking martingales]\label{thm:kellerer-regularized}
	Let $(\mu_t)_{t \in [0, 1]}$ be a weakly continuous square-integrable peacock on $\R^d$. Fix $\delta, \varepsilon > 0$ and, for each $t \in [0, 1]$, define the regularized measure $\mu^\reg_t \coloneqq \mu_t \ast \gamma^{\varepsilon (t +\delta)}$. Then there exists a strongly Markovian martingale It\^o diffusion $(M_t)_{t \in [0, 1]}$ mimicking the regularized peacock $(\mu^\reg_t)_{t \in [0, 1]}$.
	
	More precisely, there exists a measurable function $(t, x) \mapsto \sigma_t(x)$ on $[0, 1] \times \R^d$, taking values in the set of positive definite matrices such that, on any probability space $(\Omega, \F, \P)$ supporting a standard $\R^d$-valued Brownian motion $(B_t)_{t \in [0, 1]}$ and an independent random variable $\xi \sim \mu_0$, the martingale $M$ satisfies
	\begin{equation}\label{eq:mimicking-m}
		\D M_t = \sigma_t(M_t) \D B_t, \quad M_0 = \xi.
	\end{equation}
	The map $(t, x) \mapsto \sigma_t(x)^2 \coloneqq \sigma_t(x)\sigma_t(x)^\top$ is locally Lipschitz continuous in the variable $x$, uniformly in $t \in [0, 1]$ and, for every $x \in \R^d$, there exist constants $c_x, C_x > 0$ such that, for $t \in [0, 1]$, we have the bounds
	\begin{equation}
		c_x \, \id \leq \sigma_t(x)^2 \leq C_x \, \id.
	\end{equation}
	Moreover, the martingale $M$ is a Feller process.
\end{theorem}

Note that, in particular, the mimicking martingale that we construct in \Cref{thm:kellerer-regularized} is continuous and strongly Markovian.
A key ingredient in the construction of this mimicking martingale is the following result that allows us to pass to limits of martingale It\^o diffusions, the details of which are presented in \Cref{sec:compactness-Ito}.

\begin{theorem}[compactness of martingale It\^o diffusions]\label{thm:compactness}
	A set of martingale It\^o diffusions satisfying \Crefrange{it:assumption_1}{it:assumption_5} is precompact in the set of martingale It\^o diffusions with respect to convergence in finite dimensional distributions.
\end{theorem}

Our next main result is that, in dimension $d \geq 2$, mimicking martingales of the form \eqref{eq:mimicking-m} may not be unique.

\begin{theorem}[non-uniqueness of mimicking martingales]\label{thm:non-uniqueness}
	Let $(B_t)_{t \in [0, 1]}$ be a standard Brownian motion on $\R^2$ with initial law $\Law(B_0) = \eta$, where $\eta$ is rotationally invariant with finite second moment. Define a peacock $\mu$ by $\mu_t = \Law(B_t)$, for $t \in [0, 1]$.
	
	Then there exists a continuous strongly Markovian martingale diffusion $(M_t)_{t \in [0, 1]}$ of the form \eqref{eq:mimicking-m}, that is not a Brownian motion, such that $\Law(M_t) = \mu_t$, for all $t \in [0, 1]$.
\end{theorem}

We further construct a series of counterexamples in dimension $d = 4$ which show that, without regularization, \Cref{thm:kellerer-regularized} does not hold in full generality, even without imposing continuity of the mimicking martingale, let alone the It\^o diffusion property.

\begin{theorem}[necessity of regularization]\label{thm:necessity}
	There exists a weakly continuous square-integrable peacock $(\mu_t)_{t \geq 0}$ on $\R^4$ such that, for the peacock $(\mu_t \ast \gamma^t)_{t \geq 0}$, there exists no mimicking Markov martingale.
\end{theorem}

While previous authors have considered the problem of finding mimicking martingales in general dimensions, to the best of our knowledge the present work is the first to provide a multidimensional extension of Kellerer's theorem. Prior to Kellerer's work, Doob \cite{Do68} proved the existence of mimicking martingales taking values in an abstract compact space in continuous time, but notably did not consider the Markov property. More recently, Hirsch and Roynette \cite{HiRo13} proved existence for continuous-time peacocks on $\R^d$, $d \geq 1$, with right-continuous paths, again without the Markov property.

Juillet \cite{Ju16} considered generalizing Kellerer's theorem in two different directions, first showing that when a peacock on $\R$ is indexed by a two-parameter family with some partial order, mimicking martingales may not exist at all. Moreover, \cite{Ju16} provides an example of a peacock on $\R^2$ for which there exists no mimicking martingale that additionally satisfies the so-called \emph{Lipschitz Markov} property, defined in \cite[Definition 6]{Ju16}. The Lipschitz Markov property implies the Feller property and, for c\`adl\`ag processes, the strong Markov property; see \cite[Lemma 4.2]{Lo09}. The key property of the class of c\`adl\`ag Lipschitz Markov processes is compactness with respect to convergence in finite dimensional distributions, as shown in \cite[Lemma 4.5]{Lo09}. On the other hand, it is well known that the class of Markov martingales is not closed with respect to this mode of convergence; see, e.g.~\cite[Example 1]{BeHuSt16}. All proofs of Kellerer's theorem that are known to us make use of the compactness of Lipschitz Markov processes; see, e.g.~\cite{BeHuSt16,HiRoYo14,Ke72,Lo09}. In light of the result of \cite{Ju16}, the notion of Lipschitz Markovianity does not lend itself well to the higher-dimensional problem. In its place, we consider a class of Feller processes that are martingale It\^o diffusions with particular properties. We show in \Cref{thm:compactness} that this set of processes is compact with respect to convergence in finite dimensional distributions.

We have seen that, in dimension one, uniqueness holds in the class of continuous strong Markov mimicking martingales when the marginals of the peacock have convex support. \Cref{thm:non-uniqueness} shows that strong Markovianity is not sufficient to guarantee uniqueness in higher dimensions, by exhibiting a continuous two-dimensional strong Markov martingale with Brownian marginals that is not itself a Brownian motion. The question of the existence of martingales distinct from Brownian motion that have Brownian marginals goes back to Hamza and Klebaner \cite{HaKl07}, who showed that such a \emph{fake Brownian motion} with discontinuous paths exists in one dimension. As already mentioned, the culmination of this one-dimensional investigation was the construction \cite{BePaSc21a} of a continuous Markovian fake Brownian motion. Of course Brownian motion is the unique continuous \emph{strong} Markov martingale with Brownian marginals in one dimension. In two dimensions however, we show in \Cref{thm:non-uniqueness} that there exists a fake Brownian motion that is continuous and strongly Markovian.

We remark that the mimicking martingale of \Cref{thm:kellerer-regularized} is an It\^o diffusion process with Markovian diffusion coefficient. Finding mimicking martingales of this form has also received extensive interest since the work of Krylov \cite{Kr85} and Gy\"ongy \cite{Gy86}. In fact we twice apply a more recent result of Brunick and Shreve \cite{BrSh13} on mimicking Markovian diffusions in our construction in \Cref{sec:proof-of-theorem}.

For a more detailed review of the existing literature, we refer the reader to the surveys of Hirsch, Roynette and Yor \cite{HiRoYo14} and Beiglb\"ock, Pammer and Schachermayer \cite{BePaSc22}, and the references therein.

The structure of the present article is as follows. In \Cref{sec:proof-of-theorem}, we construct a strongly Markovian mimicking martingale It\^o diffusion, thus proving \Cref{thm:kellerer-regularized}. We then prove \Cref{thm:non-uniqueness} in \Cref{sec:non-uniqueness}, by providing a counterexample to uniqueness of mimicking martingales. We present further examples in \Cref{sec:necessity}, which show that existence may fail without regularization, thus proving \Cref{thm:necessity}. Finally, in \Cref{sec:compactness-Ito}, we prove the compactness result \Cref{thm:compactness} for martingale It\^o diffusions, which is key to the proof of \Cref{thm:kellerer-regularized} in \Cref{sec:proof-of-theorem}.

We introduce some notation and terminology that will be used throughout the paper. The notation $|\cdot|$ represents the Euclidean norm on $\R^d$, and $B_R$ denotes the closed ball with radius $R > 0$ centered at the origin. We denote by $\Pc_2(\R^d)$ the set of probability measures on $\R^d$ with finite second moment. We denote by $\F^X$ the natural filtration of a stochastic process $X$, enlarged as necessary to satisfy the usual conditions. For measures $\mu, \nu$, we write $\mu \preceq \nu$ to denote that $\mu$ is dominated by $\nu$ in convex order; i.e.~for any convex function $f$, $\int f \D \mu \leq \int f \D \nu$. A family of measures $(\mu_t)_{t \in I}$ is called a \emph{peacock} if it is increasing in the convex order.\footnote{The terminology \emph{peacock} was introduced by Hirsch, Profeta, Roynette and Yor \cite{HiPrRoYo11} as a pun on the French \emph{Processus Croissant pour l'Ordre Convexe} (PCOC), meaning a process increasing in convex order.} We say that a process $(X_t)_{t \in I}$ \emph{mimics} $(\mu_t)_{t \in I}$ if $\Law(X_t) = \mu_t$ for all $t \in I$.

For matrices $A, B \in \R^{d \times d}$ the notation $A \leq B$ denotes that the matrix $B - A$ is positive semidefinite. When working with matrices, we will always use the Hilbert--Schmidt norm (also known as the Frobenius norm): for $A \in \R^{d \times d}$, we write $\lVert A \rVert = \trace(AA^\top) = \sum_{i, j = 1}^d A_{ij}^2 = \sum_{i = 1}^d \lambda_i^2$, where $\lambda_i$ are the eigenvalues of $A$. We further denote the square matrix $A \coloneqq A A^\top$.

\section{Construction of a mimicking martingale}\label{sec:proof-of-theorem}

Let $d \geq 2$ and let $(\mu_t)_{t \in [0, 1]}$ be a weakly continuous peacock in $\mathcal{P}_2(\R^d)$; i.e.~$\mu_{t_0} \preceq \mu_{t_1}$ for all $t_0 \leq t_1$, $\sup_{t \in [0, 1]} \int |x|^2 \mu_t(\D x) = \int |x|^2 \mu_1(\D x) < \infty$, and $t \mapsto \int f \D \mu_t$ is continuous for any bounded continuous function $f$. Fix $\delta, \varepsilon > 0$ and define the regularized peacock $\mu^\reg$ by
\begin{equation}\label{eq:reg-def}
	\mu^\reg_t = \mu_t \ast \gamma^{\varepsilon(t + \delta)}, \quad t \in [0, 1].
\end{equation}
Note that the process $(\mu^\reg_t)_{t \in [0, 1]}$ is a peacock satisfying $\mu_t \preceq \mu^\reg_t$, for all $t \in [0, 1]$.

\begin{remark}\label{rem:time-change}
	The relevant feature of the function
	\begin{equation}\label{eq:gauss-conv}
		\varphi(t) = \varepsilon (t + \delta), \quad t \in [0, 1],
	\end{equation}
	is that $\phi(0) > 0$ and $t \mapsto \phi(t)$ is strictly increasing. It will become clear from the construction below that we can replace \eqref{eq:gauss-conv} with any such function.
	
	In this context, we also normalize the peacock $(\mu_t)_{t \in [0, 1]}$ by making a deterministic time change so that $\int |x|^2 \mu_{t+h}(\D x) - \int |x|^2 \mu_t(\D x) = h$, for $t \in [0, 1)$, $h > 0$.
	For convenience, we still take $\varphi$ as in \eqref{eq:gauss-conv} after the time change.
	
	Moreover, in place of the Gaussian family $(\gamma^{\varepsilon(t + \delta)})_{t \in [0, 1]}$, one may take another family of centered probability measures $(\eta_t)_{t \in [0, 1]} \subseteq \Pc_2(\R^d)$ that is weakly continuous and \emph{strictly increasing} in convex order. Provided that these measures have smooth densities that are uniformly bounded from below on $\R^d$, one could apply similar arguments to prove an analogue of \Cref{thm:kellerer-regularized}. However, the proof would become significantly more involved.
\end{remark}

Fix $n \in \N$ and consider the dyadics $S^n \coloneqq \{2^{-n}, 2\cdot 2^{-n}, \dotsc, 2^n\cdot 2^{-n}\}\subseteq [0,1]$. For $k \in \{0, \dotsc, 2^n\}$, denote $t^n_k \coloneqq k 2^{-n}$. We will construct a martingale It\^o diffusion that mimics $\mu^\reg$ on the dyadics $S^n$. \Cref{thm:fdd_convergence} will allow us to pass to a limit. We first construct martingale It\^o diffusions on each dyadic interval, before concatenating these intervals. This step is rather standard; cf.~\cite{HiPrRoYo11, HiRoYo14}. For our purposes it is convenient to use the concept of \emph{stretched Brownian motion} introduced in \cite{BaBeHuKa20}. We now fix $k \in \{0, \dotsc, 2^n - 1\}$ and consider the interval $[t^n_k, t^n_{k + 1})$. Recall that $B$ denotes a standard Brownian motion on $\R^d$.

\begin{lemma}\label{lem:first-diffusion}
	Let $(\mu_t)_{t \in [0, 1]}$ be a weakly continuous square-integrable peacock. Then there exists a strongly Markovian martingale diffusion $(\bar{M}^{n, k}_t)_{t \in [t^n_k, t^n_{k +1}]}$, with the representation
	\begin{equation}\label{eq:sbm}
		\D \bar{M}^{n, k}_t = \bar{\sigma}^{n, k}_t(\bar{M}^{n, k}_t) \D B_t, \quad \text{on} \quad [t^n_k, t^n_{k + 1}),
	\end{equation}
	for some measurable function $(t, x) \mapsto \bar{\sigma}^{n, k}_t(x)$, taking values in the set of positive semidefinite matrices, such that $\Law(\bar{M}^{n, k}_{t^n_k}) = \mu_{t^n_k}$ and $\Law(\bar{M}^{n, k}_{t^n_{k + 1}}) = \mu_{t^n_{k+1}}$.
\end{lemma}

\begin{proof}
	Let $(\bar{M}^{n, k})_{t \in [t^n_k, t^n_{k +1}]}$ be the stretched Brownian motion with $\Law(\bar{M}^{n, k}_{t^n_k}) = \mu_{t^n_k}$ and $\Law(\tilde{M}^{n, k}_{t^n_{k+1}}) = \mu_{t^n_{k+1}}$, as defined in \cite[Definition 1.6]{BaBeHuKa20}. By definition, $\bar M^{n, k}$ is a martingale with the representation $\D \bar M^{n, k}_t = \theta^{n, k}_t \D B_t $, for some $\F^B$-predictable process $\theta^{n,k}$ taking values in the set of positive semidefinite matrices. Moreover, by \cite[Corollary 2.5]{BaBeHuKa20}, $\bar M^{n, k}$ is a strong Markov process. Thus \Cref{lem:markov-diffusion} gives the existence of a measurable function $\bar{\sigma}^{n, k}_\cdot \colon [t^n_k, t^n_{k + 1}) \times \R^d \to \R^{d \times d}$ such that \eqref{eq:sbm} holds.
\end{proof}

We do not yet have a control on the matrix norm of $(\bar{\sigma}^{n, k}_t(\bar{M}^{n, k}))_{t \in [t^n_k, t^n_{k + 1}]}$. In order to achieve an upper bound on the diffusion matrix, we make a first convolution with a Gaussian. This will have an averaging effect and allow us to control the diffusion from above almost surely. Namely, we take a centered Gaussian random variable $\Gamma^{n, k}$ with covariance matrix $(\varepsilon[t^n_k + \delta]) \id$, independent of $\F^B$ and $\F^{\bar{M}^{n, k}}$, and define
\begin{equation}\label{eq:sum-independent-gaussian}
	\tilde M^{n, k}_t \coloneqq \bar{M}^{n, k}_t + \Gamma^{n, k}, \quad t \in [t^n_k, t^n_{k +1}].
\end{equation}
Then, for the initial law in this interval, we have $\Law(\tilde{M}^{n, k}_{t^n_k}) = \mu_{t^n_k} \ast \gamma^{\varepsilon(t^n_k + \delta)} = \mu^\reg_{t^n_k}$, and for the terminal law, we have the ordering $\Law(\tilde{M}^{n, k}_{t^n_{k+1}}) = \mu_{t^n_{k+1}} \ast \gamma^{\varepsilon(t^n_k + \delta)} \preceq \mu_{t^n_{k+1}} \ast \gamma^{\varepsilon(t^n_{k+1} + \delta)} = \mu^\reg_{t^n_{k + 1}}$, where we recall the definition of $\mu^\reg$ from \eqref{eq:reg-def}. Later we will make a second Gaussian convolution, which will allow us to also bound the squared diffusion matrix from below. We now prove that the square of the diffusion matrix obtained after the first convolution is locally bounded and locally Lipschitz.

\begin{lemma}\label{lem:markov-mixing}
	For $n \in \N$, $k \in \{0, \dotsc, 2^n\}$ and $\bar \sigma^{n, k}$ as in \eqref{eq:sbm}, define the matrix-valued function $(t, x) \mapsto \tilde{\sigma}^{n, k}_t(x)$ as the unique positive semidefinite square root of
	\begin{equation}\label{eq:sigma-mixed}
		\tilde{\sigma}^{n, k}_t(x)^2 \coloneqq \frac{\int \bar \sigma^{n, k}_t(y)^2 g^{n, k}(x - y) \bar{m}^{n, k}_t(\D y)}{\int g^{n, k}(x - y)\bar{m}^{n, k}_t(\D y)}, \quad t \in [t^n_k, t^n_{k + 1}), x \in \R^d,
	\end{equation}
	where $g^{n, k}$ is the density of a Gaussian with mean zero and covariance matrix $(\varepsilon[t^n_k + \delta])\id$, and $\bar{m}^{n, k}_t = \Law(\bar{M}^{n, k}_t)$.
	
	Then for every compact set $K \subseteq \R^d$, there exist constants $C_K, L_K$, independent of $t$, $k$ and $n$, such that
	\begin{equation}
		\lVert \tilde{\sigma}^{n, k}_t(x)^2 \rVert \leq C_K, \quad (t, x) \in [t^n_k, t^n_{k +1}] \times K,
	\end{equation}
	and $x \mapsto \tilde{\sigma}^{n, k}_t(x)^2$ is Lipschitz on $K$ with Lipschitz constant $L_K$, for all $t \in [t^n_k, t^n_{k + 1}]$.
\end{lemma}

\begin{proof}		 
	 Choose $R > 0$ such that $\int |x| \D(\mu_1 \ast \gamma^{\varepsilon(1+\delta)})(x) \leq R/2$. Applying Doob's maximal inequality and the convex ordering of the marginals gives the bound $\bar{m}^{n, k}_t(B_R) \geq \frac{1}{2}$, for all $t \in [t^n_k, t^n_{k + 1}]$ and all $k \in \{1, \dotsc, 2^n - 1\}$. Then, for an arbitrary compact set $K \subseteq \R^d$, we can bound the normalising constant in the denominator of \eqref{eq:sigma-mixed} by
	 \begin{equation}\label{eq:denominator-bound}
	 	\begin{split}
		 	\int_{\R^d} g^{n, k}(x - y)\bar{m}^{n, k}_t(\D y) & \geq \int_{B_R} g^{n, k}(x - y)\bar{m}^{n, k}_t(\D y)\\
		 	& \geq \frac{\bar{C}_K}{2}(2 \pi \varepsilon [t^n_k + \delta])^{-\frac{d}{2}}, \quad \text{for} \quad x \in K,
	 	\end{split}
	 \end{equation}
	 where $\bar{C}_K \coloneqq \inf \!\left\{\exp \{- \varepsilon^{-1} \delta^{-1} \lvert x - y \rvert^2 \} \colon \; x \in K, y \in B_R\right\}\! > 0$, independent of $t$ and $k$. We bound the numerator of \eqref{eq:sigma-mixed} in the Hilbert--Schmidt norm by
	 \begin{equation}
	 	\begin{split}
	 		\!\left \Vert \int \bar \sigma^{n, k}_t(y)^2 g^{n, k}(x - y) \bar{m}^{n, k}_t(\D y)\right \rVert & \leq (2 \pi \varepsilon [t^n_k + \delta])^{-\frac{d}{2}} \E [\lVert \bar{\sigma}^{n, k}_t(\bar{M}^{n, k}_t)\rVert^2].
	 	\end{split}
	 \end{equation}
	 Recall from \Cref{rem:time-change} that $\E[\lvert \bar{M}^{n, k}_{t + h}\rvert^2 - \lvert \bar{M}^{n, k}_{t}\rvert^2] = h$, for all $t \in [t^n_k, t^n_{k + 1})$ and $h$ sufficiently small. But, by the It\^o isometry,
	 \begin{equation}
	 	\E[\lvert \bar{M}^{n, k}_{t + h}\rvert^2 - \lvert \bar{M}^{n, k}_{t}\rvert^2] = \int_t^{t + h} \E [\lVert \bar{\sigma}^{n, k}_t(\bar{M}^{n, k}_t)\rVert^2] \D t,
	 \end{equation}
	 and so $\E [\lVert \bar{\sigma}^{n, k}_t(\bar{M}^{n, k}_t)\rVert^2] = 1$, for all $t \in [t^n_k, t^n_{k +1}]$. Altogether, we have the upper bound
	 \begin{equation}
	 	\lVert \tilde \sigma^{n, k}_t(x)^2 \rVert \leq \frac{2}{\bar{C}_K}, \quad \text{for} \quad x \in K.
	 \end{equation}
	 
	 It remains to prove continuity. To save notation in the following calculation, we suppress the dependency on $n$ and $k$, setting $g \coloneqq g^{n, k}$, $\bar{\sigma} \coloneqq \bar{\sigma}^{n, k}$, $\bar{m} \coloneqq \bar{m}^{n, k}$. Fix $t \in [t^n_k, t^n_{k + 1}]$ and $x_0, x_1 \in K$, for some compact set $K \subseteq \R^d$. Then, from the definition \eqref{eq:sigma-mixed} and the bound \eqref{eq:denominator-bound}, we calculate
	 \begin{equation}
	 	\begin{split}
		 	& \!\left\lVert \tilde{\sigma}^{n, k}_t(x_1)^2 - \tilde{\sigma}^{n, k}_t(x_0)^2 \right \rVert\\
		  	& \qquad \leq 4(2 \pi \varepsilon [t^n_k + \delta])^d\bar{C}_K^{-2}\!\left\lVert \int g(x_0 - y) \bar{m}_t(\D y) \int \bar \sigma_t(y)^2 g(x_1 - y) \bar{m}_t(\D y) \right .\\
		 	& \qquad \qquad \qquad \qquad \qquad \!\left. - \int g(x_1 - y) \bar{m}_t(\D y) \int \bar \sigma_t(y)^2g(x_0 - y) \bar{m}_t(\D y) \right\rVert\\
		 	& \qquad \leq 4(2 \pi \varepsilon [t^n_k + \delta])^d\bar{C}_K^{-2} \!\left \lvert \int g(x_0 - y) \bar{m}_t(\D y)  \right \rvert \!\left \lVert \int \!\left[ g(x_1 - y) - g(x_0 - y)\right]\! \bar \sigma_t(y)^2 \bar{m}_t(\D y) \right \rVert\\
		 	& \qquad \qquad + 4(2 \pi \varepsilon [t^n_k + \delta])^d\bar{C}_K^{-2} \!\left \lvert \int \!\left[ g(x_0 - y) - g(x_1 - y) \right]\! \bar{m}_t(\D y) \right\rvert \!\left \lVert \int g(x_0 - y) \bar \sigma_t(y)^2 \bar{m}_t(\D y) \right \rVert. 
	 	\end{split}
	 \end{equation}
	 As in the proof of the upper bound, note that $\lVert \int \bar \sigma_t(y)^2 \bar{m_t}(\D y) \rVert \leq 1$, by \Cref{rem:time-change}. We also see that each constant $(2 \pi \varepsilon [t^n_k + \delta])^\frac{d}{2}$ cancels with a normalising constant from $g = g^{n, k}$. The Gaussian density $g$ is Lipschitz with Lipschitz constant $L$ that can be taken independent of $t, n, k$. Together, we find that
	 \begin{equation}
	 	\begin{split}
	 		\!\left\lVert \tilde{\sigma}^{n, k}_t(x_1)^2 - \tilde{\sigma}^{n, k}_t(x_0)^2 \right \rVert & \leq 8 \bar C_K^{-2} L |x_1 - x_0|,
	 	\end{split}
	 \end{equation}
	 as required.
\end{proof}

\begin{lemma}\label{lem:brunick-shreve}
	For each $n \in \N$, there exists a probability space $(\Omega^n, \F^n, \P^n)$ on which an $\R^d$-valued Brownian motion $B^n$ and $\R^d$-valued martingales $(\hat M^{n, k}_t)_{t \in [t^n_k, t^n_{k + 1}]}$ are defined such that
	\begin{equation}
		\D \hat{M}^{n, k}_t = \tilde{\sigma}^{n, k}_t(\hat{M}^{n, k}_t) \D B^n_t \quad \text{on} \quad [t^n_k, t^n_{k + 1}],
	\end{equation}
	where $\tilde \sigma^{n, k}$ is defined as in \eqref{eq:sigma-mixed}, for each $k \in \{0, \dotsc, 2^n\}$.
	Moreover, for each $k \in \{0, \dotsc, 2^n\}$ and $t \in [t^n_k, t^n_{k +1}]$, $\Law(\hat{M}^{n, k}_t) = \Law(\tilde{M}^{n, k}_t) = \bar{m}^{n, k}_{t} \ast \gamma^{\varepsilon(t^n_k + \delta)}$, and in particular, $\Law(\hat{M}^{n, k}_{t^n_{k + 1}}) = \mu_{t^n_{k +1}} \ast \gamma^{\varepsilon (t^n_k + \delta)}$.
\end{lemma}

\begin{proof}
	Fix $n \in \N$. From \eqref{eq:sum-independent-gaussian}, we have
	 \begin{equation}
	 	\D \tilde M^{n, k}_t = \bar \sigma^{n, k}_t(\bar M^{n, k}_t)\D B_t = \bar \sigma^{n, k}_t(\tilde M^{n, k}_t - \Gamma^{n,k})\D B_t.
	 \end{equation}
	 By the result on mimicking diffusions of Brunick and Shreve \cite[Corollary 3.7]{BrSh13}, there exist measurable maps $\hat \sigma^{n, k} \colon [t^n_k, t^n_{k + 1}] \times \R^d \to \R^d \times \R^d$ taking values in the set of positive semidefinite matrices such that, for any $t \in [t^n_k, t^n_{k + 1}]$,
	 \begin{equation}
	 	\hat \sigma^{n, k}_t(\tilde M^{n, k}_t)^2 = \E\!\left[\bar \sigma^{n, k}_t(\tilde M^{n, k}_t - \Gamma^{n,k})^2 \mid \tilde M^{n, k}_t\right]\!
	 \end{equation}
	 almost surely, for each $k \in \{0, \dotsc, 2^n\}$. Again by \cite[Corollary 3.7]{BrSh13}, there exists a probability space $(\Omega^n, \F^n, \P^n)$ supporting a Brownian motion $B^n$ and martingales $\hat M^{n,k}$ such that, for any $t \in [t^n_k, t^n_{k + 1}]$,
	 \begin{equation}
	 	\hat M^{n,k}_t = \hat M^{n,k}_{t^n_k} + \int_{t^n_k}^t \hat \sigma^{n, k}_s(\hat M^{n, k}_s) \D B^n_s,
	 \end{equation}
	 and $\Law(\hat M^{n, k}_t) = \Law(\tilde M^{n, k}_t)$, for each $k \in \{0, \dotsc, 2^n\}$. We conclude by computing that, for all $t \in [t^n_k, t^n_{k + 1}]$ and $\Law(\tilde M^{n,k}_t)$-almost every $x \in \R^d$,
	 \begin{equation}
	 	\hat \sigma^{n,k}_t(x)^2 = \tilde \sigma^{n, k}_t(x)^2,
	 \end{equation}
	 for $\tilde \sigma^{n,k}$ defined in \eqref{eq:sigma-mixed}.
\end{proof}

For each $n \in \N$ and $k \in \{0, \dotsc, 2^n\}$, we now have a martingale $\hat{M}^{n, k}$ on $[t^n_k, t^n_{k + 1}]$, whose squared diffusion matrix $(t, x) \mapsto \tilde \sigma^{n, k}_t(x)^2$ is bounded from above on compact sets in $[0, 1] \times \R^d$. To achieve a lower bound, we divide the interval $[t^n_k, t^n_{k + 1}]$ in half and time-change the martingale by a factor of two in the first half of the interval. In the second half of the interval, we shall simply add a Brownian motion with an appropriately scaled covariance. This gives us a second Gaussian convolution to arrive at the measure $\mu^\reg_{t^n_{k +1}} = \mu_{t^n_{k + 1}} \ast \gamma^{\varepsilon(t^n_{k +1} + \delta)}$, rather than $\mu_{t^n_{k + 1}} \ast \gamma^{\varepsilon(t^n_k + \delta)}$, at the terminal time $t^n_{k+1}$.

Define the function $(t, x) \mapsto \sigma^{n, k}_t(x)$ by
\begin{equation}\label{eq:sigma-splitting-interval}
	\sigma^{n, k}_t(x) \coloneqq
	\begin{cases}
		\sqrt{2}\; \tilde{\sigma}^{n, k}_{t^n_k + 2(t - t^n_k)}(x), & t \in [t^n_k, t^n_k + 2^{-(n+1)}),\vspace{1ex}\\
		\sqrt{2\varepsilon} \, \id, & t \in [t^n_k + 2^{-(n + 1)}, t^n_{k+1}].
	\end{cases}
\end{equation}
Take a random variable $M^{n, k}_{t^n_k}$ on $(\Omega^n, \F^n, \P^n)$ with $\Law(M^{n, k}_{t^n_k}) = \mu^\reg_{t^n_k}$ and, for $t \in [t^n_k, t^n_{k +1}]$, define
\begin{equation}
	M^{n, k}_t = M^{n, k}_{t^n_k} + \int_{t^n_k}^t \sigma^{n, k}_s(M^{n, k}_s) \D \hat B_s.
\end{equation}
Then we see that
\begin{equation}\label{eq:dyadic-law}
	\Law(M^{n, k}_{t^n_{k+1}}) = (\mu_{t^n_{k + 1}} \ast \gamma^{\varepsilon (t^n_k + \delta)}) \ast \gamma^{\varepsilon2^{-n}} = \mu_{t^n_{k + 1}} \ast \gamma^{\varepsilon (t^n_{k+1} + \delta)} = \mu^\reg_{t^n_{k + 1}},
\end{equation}
and we have the lower bound
\begin{equation}\label{eq:lower-bound}
	\sigma^{n, k}_t(x)^2 \geq 2 \varepsilon \, \id, \quad t \in [t^n_k + 2^{-(n + 1)}, t^n_{k + 1}], \; x \in \R^d.
\end{equation}

We next paste together the martingales defined on each interval. Define $\sigma^{n} \colon [0, 1] \times \R^d \to \R^{d \times d}$ by
\begin{equation}
	\sigma^{n}_t(x) \coloneqq \sum_{k = 0}^{2^n - 1} \sigma^{n, k}_t(x) \mathds{1}_{(t^n_k, t^n_{k +1}]}(t), \quad t \in [0, 1], \; x \in \R^d.
\end{equation}
We arrive at the following proposition.

\begin{proposition}\label{prop:pasting-intervals}
	For each $n \in \N$, there exists a probability space $(\Omega^n, \F^n, \P^n)$ on which there exists a Brownian motion $B^n$ and a martingale diffusion $(M^{n}_t)_{t \in [0, 1]}$ satisfying
	\begin{equation}
		\D M^{n}_t = \sigma^{n}_t(M^{n}_t) \D B_t,
	\end{equation}
	with $\Law(M^{n}_r) = \mu^\reg_r$, for all $r \in S^n$. The family $\{|M^n_1|^2 \colon \; n \in \N\}$ is uniformly integrable.
	
	Moreover, $(t, x) \mapsto \sigma^n_t(x)^2$ is locally bounded and locally Lipschitz in $x$, uniformly in $t \in [0, 1]$ and $n \in \N$, and
	\begin{equation}\label{eq:lower-bound-unif}
		\sigma^{n}_t(x)^2 \geq 2 \varepsilon \, \id,
	\end{equation}
	for any $x \in \R$ and $t \in \bigcup_{k = 0}^{2^n - 1}[t^n_k + 2^{-(n + 1)}, t^n_{k + 1}]$, $n \in \N$.
\end{proposition}

\begin{proof}
	By \Cref{lem:brunick-shreve}, for each $n \in \N$, we can find a probability space $(\Omega^n, \F^n, \P^n)$ supporting a Brownian motion $B^n$ such that we can define a random variable $M^n_0$ with $\Law(M^n_0) = \mu^\reg_0$ and, for $t \in [0, 1]$,
	\begin{equation}
		M^n_t = M^n_0 + \int_0^t \sigma^n_s(M^n_s) \D B^n_s.
	\end{equation}
	For each dyadic $t^n_k \in S^n$, $k \in \{0, \dotsc, 2^n\}$, $n \in \N$, we have $\Law(M^{n}_{t^n_k}) = \mu^\reg_{t^n_k}$ by \eqref{eq:dyadic-law}. In particular, $\Law(M^n_1) = \mu^\reg_1 \in \Pc_2(\Omega)$, for all $n \in \N$, which implies uniform integrability of $\{|M^n_1|^2 \colon \; n \in \N\}$.
	
	Let $K \subseteq \R^d$ be a compact set and recall the bound $C_K$ and local Lipschitz constant $L_K$ from \Cref{lem:markov-mixing}. Since these constants depend only on the compact set $K$, it follows that
	\begin{equation}
		\sigma^n_t(x)^2 \leq 2 (C_K \vee \varepsilon) \id, x \in K,
	\end{equation}
	uniformly in $t \in [0, 1]$ and $n \in \N$, and $(t, x) \mapsto \sigma^n_t(x)$ is locally Lipschitz in $x$ with Lipschitz constant $2 L_K$ on the set $K$, uniformly in $t \in [0, 1]$ and $n \in \N$. Finally, the lower bound \eqref{eq:lower-bound-unif} follows immediately from \eqref{eq:lower-bound}.
	\end{proof}

The final step is to find a limiting martingale that mimics the peacock $\mu^\reg$ at every time $t \in [0, 1]$. In \Cref{sec:compactness-Ito}, we will prove a result on compactness of It\^o diffusions with respect to convergence in finite dimensional distributions, which is tailor-made for the present application. This plays an analogous role to compactness of Lipschitz Markov processes in the one-dimensional setting of \cite{Lo09}. We now use our compactness result to allow us to pass to a limit and complete the proof of \Cref{thm:kellerer-regularized}.

\begin{proof}[Proof of \Cref{thm:kellerer-regularized} (admitting \Cref{thm:fdd_convergence})]
	Take the sequence of functions $\sigma^n\colon [0, 1] \times \R^d \to \R^{d \times d}$ and martingales $M^n$, $n \in \N$, to be as in \Cref{prop:pasting-intervals}. Then \Crefrange{it:assumption_1}{it:assumption_5} are satisfied. By \Cref{thm:fdd_convergence}, there exists a function $(t, x) \mapsto \sigma_t(x)$, such that $\sigma^2$ is locally Lipschitz in $x$, uniformly in $t \in [0, 1]$, and $M^n$ converges in finite dimensional distributions to $M$, the unique strong solution of the SDE $\D M_t = \sigma_t(M_t) \D B_t$, with $\Law(M_0) = \mu^\reg_0$. For each $n \in \N$, we have that $\Law(M^n_t) = \mu^\reg_t$, for any dyadic $t \in S^n$. Therefore, taking the limit in finite dimensional distributions, we have $\Law(M_t) = \mu^\reg_t$ for all $t \in [0, 1]$. That is, $M$ is a mimicking martingale for the regularized peacock $\mu^\reg$.
	
	From the conclusion of \Cref{thm:fdd_convergence}, we also obtain the required bounds on $\sigma^2$. That is, for each $x \in \R^d$, there exist constants $c_x, C_x > 0$ such that, uniformly in $t \in [0, 1]$, we have
	\begin{equation}\label{eq:sigma-bounds}
		c_x \, \id \leq \sigma_t(x)^2 \leq C_x \, \id.
	\end{equation}
	It remains to verify the Feller property of $M$. The law of $(M_t)_{t \in [0, 1]}$ is a solution of the associated martingale problem of Stroock and Varadhan \cite[Chapter 6]{StVa79}. Moreover, we proved above that the diffusion coefficient $\sigma$ satisfies the bounds \eqref{eq:sigma-bounds}, and that $\sigma^2$ is locally Lipschitz in $x$, uniformly in $t \in [0, 1]$. Under these conditions, \cite[Theorem 10.1.3]{StVa79} implies that the martingale problem admits at most one solution. We now have that the martingale problem is well posed and, by applying \cite[Corollary 10.1.4]{StVa79}, the unique solution has the Feller property. We conclude that the martingale $(M_t)_{t \in [0, 1]}$ also has the strong Markov property.
\end{proof}

\section{Non-uniqueness}\label{sec:non-uniqueness}
We shall show that, even for the simple example of a two-dimensional Brownian motion, uniqueness does not hold in the class of continuous strong Markov mimicking martingales. In other words, the one-dimensional uniqueness result of Lowther \cite{Lo08b} cannot be extended directly to higher dimensions.

Considering the problem of mimicking the marginals of a standard two-dimensional Brownian motion, the Brownian motion itself is of course a continuous strong Markov martingale with the required marginals. In order to disprove uniqueness, we seek another mimicking process with these properties. We will thus construct a two-dimensional continuous \emph{fake Brownian motion} that is strongly Markovian.

\begin{proposition}
	For every peacock $\mu = (\mu_t)_{t \in [0, 1]}$ on $\R^2$ defined as in \Cref{thm:non-uniqueness}, there exist two distinct continuous strong Markov martingale diffusions mimicking $\mu$.
\end{proposition}
\begin{proof}
	Let $B$ be a standard $2$-dimensional Brownian motion started in some rotationally invariant law $\eta \in \mathcal{P}_2(\R^2) \setminus \{0\}$, and write $(\mu_t)_{t \in [0, 1]}$ for its marginals. Then $B$ is a continuous strong Markov martingale mimicking $(\mu_t)_{t \in [0, 1]}$. We now construct a continuous strong Markov martingale $M$ that mimics $(\mu_t)_{t \in [0, 1]}$ and is not itself a Brownian motion.
	
	For any $x \in \R^2$, let us denote $x^\perp \coloneqq (-x_2, x_1)^\top$, so that $x \cdot x^\perp = 0$ and $\lvert x \rvert = \lvert x^\perp \rvert$. Let $W$ be a standard $\R$-valued Brownian motion and consider the SDE
	\begin{equation}\label{eq:sde-symmetric}
		\D M_t = \frac{1}{\lvert M_t \rvert} (M_t + M^\perp_t) \D W_t; \quad \Law(M_0) = \eta.
	\end{equation}
	It is shown in \cite[Proposition 3.2]{CoRo22} that any solution of \eqref{eq:sde-symmetric} almost surely does not hit the origin. Therefore, by applying standard arguments for SDEs with Lipschitz coefficients, one can show that there exists a unique strong solution $M$ of this SDE that is a continuous strong Markov martingale. A simulated trajectory of $M$ is shown in \Cref{subfig:fake-bm}.
	
	 By \cite[Proposition 3.2]{CoRo22} again, the radius of $M$, denoted by $R_t = \lvert M_t \rvert$ for all $t \geq 0$, is a $2$-dimensional Bessel process satisfying
	\begin{equation}
		R_t = W_t + \frac{1}{2 R_t}\D t, \quad t > 0; \quad \Law(R_0) = \Law(\lvert M_0\rvert).
	\end{equation}
	Hence the radius of $M$ coincides with the radius of the $2$-dimensional Brownian motion $B$ in law; see, e.g.~\cite[Chapter XI]{ReYo99}. Moreover, the marginals of both the processes $M$ and $B$ have rotational symmetry. Hence we conclude that these marginals coincide. However, we can see that $M$ is not itself a $2$-dimensional Brownian motion, since the components of $M$ in the two coordinate directions are not independent. We have thus shown that there exist at least two distinct continuous strong Markov martingales that mimic the marginals $(\mu_t)_{t \in [0, 1]}$.
\end{proof}

\begin{remark}
	Note that $M$ solves the SDE \eqref{eq:sde-symmetric} if and only if the time-changed process $(X^\lambda_t)_{t \in [0, 1]} \coloneqq (M_{\lambda^2 t})_{t \in [0, 1]}$ solves
	\begin{equation}\label{eq:sde-lambda}
		\D X^\lambda_t = \frac{1}{\lvert X^\lambda_t \rvert}(\lambda X^\lambda_t + \sqrt{1 - \lambda^2} (X^\lambda_t)^\perp) \D W_t; \quad X_0 = x_0,
	\end{equation}
	with $\lambda = \frac{\sqrt{2}}{2}$. The SDE \eqref{eq:sde-lambda} with $\lambda \in [0, 1]$ is studied by Cox and Robinson \cite{CoRo21, CoRo22} and by Robinson \cite{Ro20}. For $\lambda = 1$, the martingale solving \eqref{eq:sde-lambda} acts as a one-dimensional Brownian motion on a fixed line through the origin --- see \Cref{subfig:rad}. For $\lambda = 0$, the martingale follows what is dubbed \emph{tangential motion} in \cite{CoRo21, CoRo22}. In this case, the process moves on a tangent to its current position, increasing the radius of the process deterministically --- see \Cref{subfig:tang}. Such a martingale already appeared in \cite{Fe18} and \cite{LaRu20} in the context of stochastic portfolio theory. Cox and Robinson \cite[Theorem 1.1]{CoRo22} showed that there is no strong solution of \eqref{eq:sde-lambda} with $\lambda = 0$ started from the origin, i.e.~$\eta = \delta_0$, drawing parallels with famous one-dimensional example of Tsirelson \cite{Ts76} and the circular Brownian motion of \'Emery and Schachermayer \cite{EmSc99}. In fact \cite[Theorem 1.2]{CoRo22} also shows that there is no strong solution of the SDE \eqref{eq:sde-lambda} started from the origin for any $\lambda \in [0, 1)$.
	 
	 We will repeatedly refer to the SDE \eqref{eq:sde-lambda} with $\lambda = 0$ in the examples of \Cref{sec:necessity} below.
	 \end{remark}

\begin{figure}[h]
	\begin{subfigure}{0.3\textwidth}
		\centering
		\includegraphics[width = \textwidth]{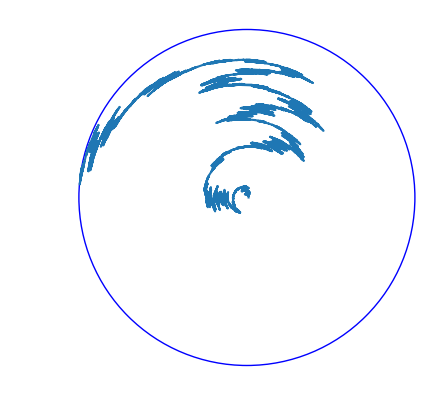}\subcaption{$\lambda = 0$,}\label{subfig:tang}
	\end{subfigure}
	\begin{subfigure}{0.3\textwidth}
		\centering
		\includegraphics[width = 0.95\textwidth]{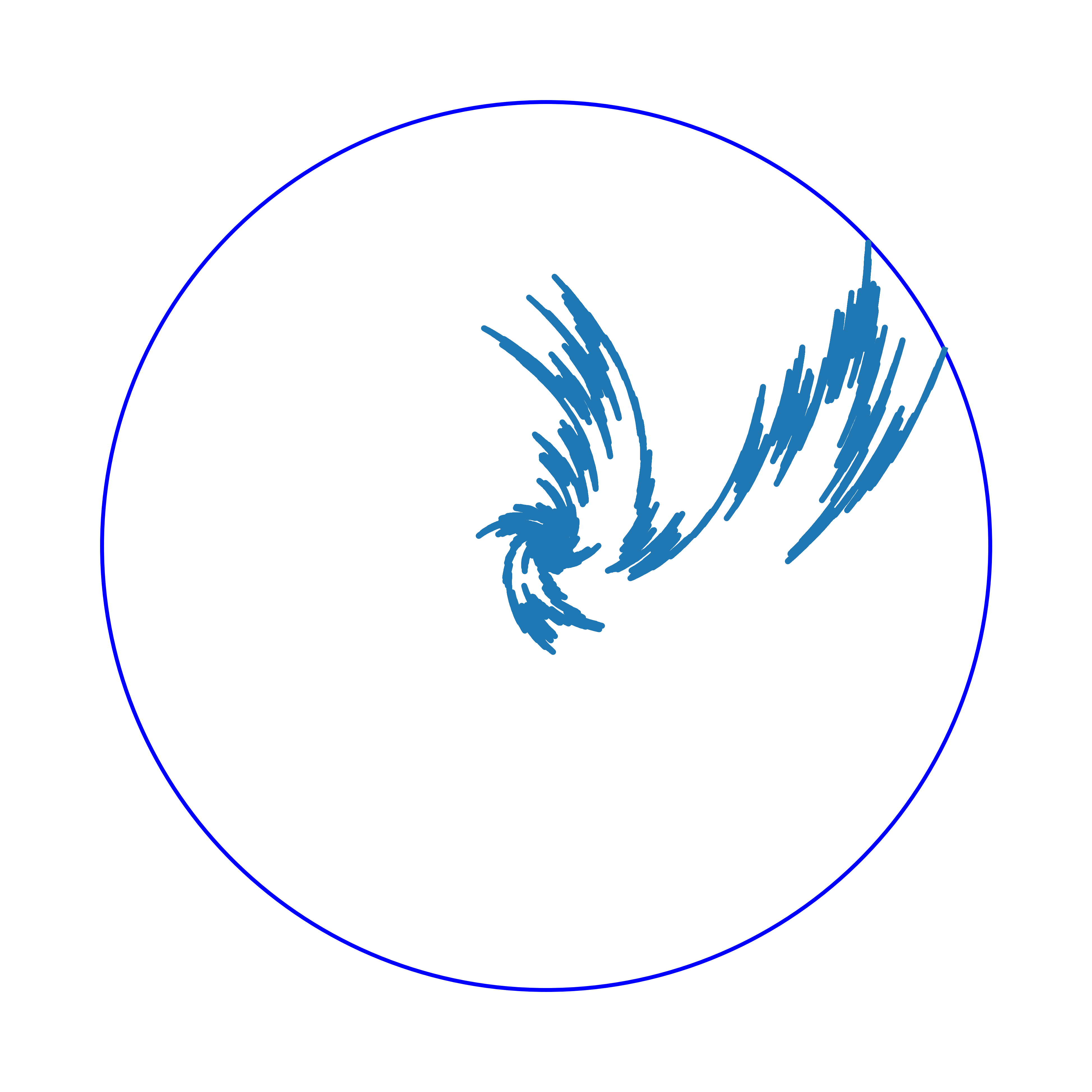}	\subcaption{fake Brownian motion,}\label{subfig:fake-bm}
	\end{subfigure}
	\begin{subfigure}{0.3\textwidth}
		\centering
		\includegraphics[width = \textwidth]{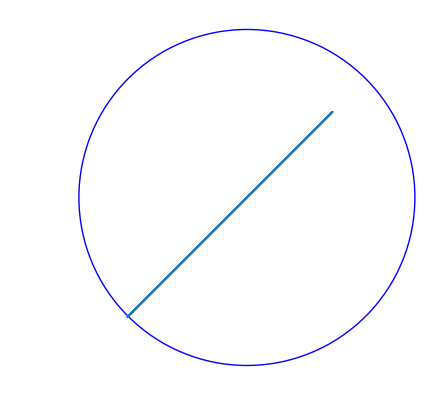}	\subcaption{$\lambda = 1$.}\label{subfig:rad}
	\end{subfigure}
	\caption{Simulations of the solution $X^\lambda$ of \eqref{eq:sde-lambda}, up to the first exit of a ball, for different values of $\lambda$. \Cref{subfig:tang} and \Cref{subfig:rad} show the extreme behaviours within the class of martingales $\{X^\lambda\colon \; \lambda \in [0, 1]\}$ (as already appeared in \cite{CoRo21}). \Cref{subfig:fake-bm} shows the midpoint between these cases, where we set $\lambda = \frac{\sqrt{2}}{2}$ then rescale time so that the martingale mimics the marginals of a Brownian motion.}\label{fig:lambda}
\end{figure}

\section{Necessity of regularization}\label{sec:necessity}

In this section, we will construct a series of counterexamples, showing that a mimicking Markov martingale may not exist without the regularization of \Cref{thm:kellerer-regularized}. We present the examples in increasing order of complexity, first showing that there may not exist a continuous mimicking Markov martingale. We then remove the continuity assumption, and finally add some (partial) regularization, in both cases showing that mimicking Markov martingales may not exist.

The following examples build on the SDE \eqref{eq:sde-lambda}, started from the origin, with $\lambda = 0$; i.e.
\begin{equation}\label{eq:cbm}
	 \D X_t = \frac{1}{|X_t|} X^\perp_t \D W_t, \quad X_0 = 0.
\end{equation}
We recall some important properties of \eqref{eq:cbm}.

\begin{remark}\label{rem:cbm}
	There exists a weak solution of \eqref{eq:cbm} by \cite[Theorem 4.3]{LaRu20}. Moreover, \cite[Theorem 1.1]{CoRo22} shows that, at any time $t \in (0, 1]$, the law of a weak solution $X$ is a uniform measure on the circle of radius $\sqrt{t}$, and so uniqueness in law holds for \eqref{eq:cbm}. In particular, a weak solution $X$ has \emph{deterministically increasing} radius
	\begin{equation}\label{eq:radius}
		|X_t| = \sqrt{t}, \quad t \in [0, 1].
	\end{equation}
\end{remark}

We construct each of the below examples\footnote{We thank Nicolas Juillet who suggested similar examples to the third named author.}  on $\R^4 \equiv \X^1 \times \X^2$, where $\X^1, \X^2$ are copies of $\R^2$. 
We also denote $S^1_t \coloneqq \{(x_1, x_2) \in \X^1 \colon \; x_1^2 + x_2^2 = t\} \times \{(0, 0)\}$, $S^2_t \coloneqq \{(0, 0)\} \times \{(x_3, x_4) \in \X^2 \colon \; x_1^2 + x_2^2 = t\}$, and $S^i \coloneqq \cup_{t \ge 0} S^i_t$ for $i = 1,2$.
Note that $S^1$ and $S^2$ only intersect at the origin.

We emphasise that, throughout the following sections, the \emph{usual conditions} of right-continuity and completeness are in force for all filtrations that we consider, and $\sigma(U)$ denotes the completion of the sigma-algebra generated by a given random variable $U$; see \Cref{rem:usual-conditions} for a discussion of these conditions.

\subsection{The continuous case}

Fix a probability space $(\Omega, \F, \P)$ on which we can define two independent copies $M^1, M^2$ of the weak solution of \eqref{eq:cbm}, as well as an independent Bernoulli($0.5$) random variable $\xi$.
Now define a process $X$ taking values in $\R^4$ by
\begin{equation}\label{eq:r4-simple}
	X_t = \begin{cases}
		(M^1_t, 0), \; & \xi = 0,\\
		(0, M^2_t), & \xi = 1,
	\end{cases}\quad t \in [0, 1],
\end{equation}
and write $\mu_t = \Law(X_t)$. Thus, following \Cref{rem:cbm}, the measure $\mu_t$ is a uniform measure on $S^1_t \cup S^2_t \subset \R^4$. Note that the process $X$ is a martingale, and so $\mu$ is a peacock.

\begin{proposition}\label{prop:ex-continuous}
	There exists a peacock $\mu$ on $\R^4$ such that there does not exist any continuous Markov process mimicking $\mu$.
\end{proposition}

\begin{proof}
	Let $(\mu_t)_{t \in [0, 1]} = (\Law(X_t))_{t \in [0, 1]}$, where $X$ is defined by \eqref{eq:r4-simple}. Suppose that there exists a continuous process $Y$ that mimics $\mu$. We will show that $Y$ is not Markovian at time $0$.
	
	By definition of the peacock $\mu$ and continuity of the paths, we have
	\begin{equation}
		\P\!\left[ \{ Y_t \in S^1_t, \; \forall t \in [0, 1] \} \cup \{  Y_t \in S^2_t, \; \forall t \in [0, 1] \} \right]\! = 1.
	\end{equation}
	In particular, for $t_0 > 0$ the events $A^1 \coloneqq \{ Y_t \in S_t^1, \; \forall t \in [0,1]\}$ and $\{ Y_{t_0} \in S^1_{t_0} \} \in \F^Y_{t_0}$ differ only by a null set, hence $\P[A^1] = 1/2$.
    On the one hand, we have by completeness and right-continuity of the filtration that $A^1 \in \F^Y_0$.
    On the other hand, $A^1$ can not be in $\sigma(Y_0)$ since the former has probability $1/2$ whereas the latter is the completion of the trivial sigma-algebra.
    Now define a function $f\colon \R^4 \to \R$ by $f(x) = \sqrt{x_3^2 + x_4^2}$. Then, for any $t \in (0, 1)$ we find
	\begin{equation}
		\frac{\sqrt{t}}{2} = \E [f(Y_t) \mid \sigma(Y_0)] \neq \E[f(Y_t) \mid \F^Y_0] =
		\begin{cases}
			0 & \text{on}\; A^1,\\
			\sqrt{t} & \text{on} \; \Omega \setminus A^1.
		\end{cases}
	\end{equation}
	 We conclude that $Y$ is not Markovian at time $0$.
\end{proof}

\begin{remark}\label{rem:initial-law}
	We clarify that the assumption that the peacock $\mu$ starts in $\mu_0 = \delta_0$ does not play a fundamental role in \Cref{prop:ex-continuous}, nor will it in \Cref{prop:ex-discontinuous,prop:ex-regularized}. Indeed, let us embed the space $\R^4 = \X^1 \times \X^2$ into $\R^5 = \R \times \X^1 \times \X^2$. Let $\bar X_0$ be any square-integrable random variable taking values in $\R \times \{(0, 0)\} \times \{(0, 0)\}$. Setting $\bar X_t = \bar X_0 + X_t$, for $t \in [0, 1]$, where $X$ is defined as in \eqref{eq:r4-simple} or \eqref{eq:r4-dyadic}, we obtain examples with the same features as required in \Cref{prop:ex-continuous,prop:ex-discontinuous,prop:ex-regularized}, without imposing that the initial law of the peacock is a Dirac measure.
\end{remark}

\begin{remark}\label{rem:ex-cadlag}
		For the peacock $\mu$ defined via \eqref{eq:r4-simple}, the continuity assumption in \Cref{prop:ex-continuous} is required in order to show non-existence of mimicking Markov processes.
        In the following we construct a c\`adl\`ag strong Markov martingale $X$ that mimics the peacock $\mu$. 
        The process behaves similarly to a compensated Poisson process:
        at time $t$ a particle $X_t$ starting either in $x = (x_1,x_2,0,0) \in \R^4$ or $x = (0,0,x_3,x_4) \in \R^4$ drifts in the direction $x$ with speed $|x|^2$.
        This drift is compensated with jumps of rate $\frac{1}{2|x|^2}$ to a uniform distribution on $S^2_{|x|^2}$ (resp.\ $S^1_{|x|^2}$).
        For $t \in \R_+$ define the rate function $\lambda$ and its anti-derivative $\Lambda$ by
        \begin{equation}
            \lambda_t \coloneqq \frac1{2t} \quad\mbox{and}\quad \Lambda_t \coloneqq \frac12 \log(t).
        \end{equation}
        To construct the process $X$, we first consider the peacock $(\mu_t)_{t \in [t_0,1]}$ where $t_0 \in (0,1)$ and define a mimicking process $X^{t_0}$.
        To this end, let $(\xi_n)_{n \in \N}$, $(U_n)_{n \in \N}$, and $(V_n)_{n \in \N}$ be families of independent random variables such that $\xi_n \sim \exp(1)$, $U_n \sim \text{Unif}(S^1_1)$, and $V_n \sim \text{Unif}(S^2_1)$.
        Given that $X_{t_0}^{t_0} = x \in S^1_{|x|^2}$, $U_0 \coloneqq x / |x|$, and $t \in (t_0,1]$, we set
        \begin{equation}
            X_t^{t_0} = \sqrt{t}\sum_{n \in \N \cup \{ 0 \} } \mathbbm 1_{\{\sum_{k = 1}^n \xi_k \le \Lambda_t - \Lambda_{t_0} < \sum_{k = 1}^{n + 1} \xi_k\}} \!\left( \mathbbm 1_{2 \mathbb Z}(n) V_n + \mathbbm 1_{2\mathbb Z + 1}(n) U_n \right)\!,
        \end{equation}
        where we use the convention that the sum over an empty index set is $-\infty$.
        Similarly, when starting in $S^2_{|x|^2}$, we define $X_t^{t_0}$ analogously to the displayed equation above but with the roles of odd and even integers reversed.
        It is straightforward to show that this process is a Feller process and thus has the strong Markov property.
        
        By \Cref{lem:app-rem-1}, which we postpone to \Cref{app:B}, $X^{t_0}$ mimics $(\mu_t)_{t \in [t_0, 1]}$ and there exists a process $X$ with the property that, for any $t_1 \in (0, 1]$,
         \begin{equation}
            \label{eq:rem.comp_Poisson.consistent}
            (X_t)_{t \in [t_1,1]} \sim (X^{t_1}_t)_{t \in [t_1,1]}.
        \end{equation}
        Hence $X$ mimics $\mu$. We deduce from \eqref{eq:rem.comp_Poisson.consistent} and the strong Markov property of $X^{t_1}$ that $X$ has the strong Markov property for stopping times $\tau$ with $\tau \ge t_1$ and $t_1 > 0$. By \Cref{lem:app-rem-2}, we also have that $X$ is Markovian at time $0$. Hence, by Lemma \ref{lem:strong_Markov} below, $X$ is a strong Markov process.
\end{remark}

\begin{remark}\label{rem:usual-conditions}
	Two referees, to whom we are indebted for their excellent remarks, have asked us to comment on the assumption that the usual conditions of right-continuity and completeness hold for the filtration $(\F^Y_t)_{t \in [0, 1]}$ considered in \Cref{prop:ex-continuous}. This example depends crucially on the assumption of right-continuity. Indeed, if we consider the natural filtration $(\F^0_t)_{t \in [0, 1]}$ generated by the process $Y$ constructed in the proof of \Cref{prop:ex-continuous} without any augmentation, the sigma-algebra $\F^0_0$ is trivial and  the process $Y$ is Markovian with respect to $\F^0$.
		
	To address this issue, we start by referring to the one-dimensional result of Kellerer \cite{Ke72}. As noted in the introduction, the set of Markov martingales is not closed with respect to convergence in finite dimensional distributions. The original proof of \cite[Theorem 3]{Ke72} therefore uses a stronger notion, which has subsequently been termed the \emph{Lipschitz Markov} property. In fact Kellerer shows the existence of a one-dimensional Lipschitz Markov mimicking martingale. Lowther \cite[Lemma 4.5]{Lo09} shows that the set of c\`adl\`ag Lipschitz Markov processes is closed with respect to convergence in finite dimensional distributions, and \cite[Lemma 4.2]{Lo09} shows that a c\`adl\`ag Lipschitz Markov process has the strong Markov property with respect to the augmented filtration. In fact, as observed in \cite[Section 3.3]{BeHuSt16}, the Lipschitz Markov property implies the Feller property. Then, for a right continuous process, strong Markovianity with respect to the augmented filtration follows from Liggett \cite[Theorem 3.3]{Li10}. Moreover, the completed natural filtration of a Feller process is right-continuous. The notion of Lipschitz Markov martingales has been used in all known proofs of the one-dimensional Kellerer theorem; see, e.g.~\cite{BeHuSt16,HiRoYo14}. In the higher dimensional setting, the Lipschitz Markov property no longer holds in general. However, we argue that it is still natural to consider Feller processes. \Cref{thm:kellerer-regularized} shows the existence of a mimicking martingale that is a Feller process in the regularised case, and \Cref{prop:ex-continuous,prop:ex-discontinuous,prop:ex-regularized} give examples for which, in particular, no Feller mimicking martingale exists.
	
	We next refer to the use of the usual conditions in the literature, which are now a standard assumption; see, e.g.~\cite{RoWi00}. The usual conditions, or \emph{conditions habituelles}, of completeness and right-continuity of filtrations are fundamental to the \emph{th\'eorie g\'en\'erale} of semi-martingales developed by the Strasbourg school; see Dellacherie and Meyer \cite{DMa,DMb}. In the context of Markov processes, Dellacherie and Meyer \cite{DMd} work predominantly with processes satisfying the strong Markov and Feller properties, and they once again work under the usual conditions. These conditions were already present in the work of Blumenthal and Getoor \cite{BlGe68} and Getoor \cite{Ge75}, where the authors consider strong Markov \emph{right processes}; see also the later references of Dellacherie and Meyer \cite{DMe}, Liggett \cite{Li10}, and Sharpe \cite{Sh88}.
	
	To give additional motivation for the use of the usual conditions, we finally put ourselves into the following financial context. Suppose that $(X_t)_{t \in [0, 1]}$ models the price of a stock (or four stocks, as in \Cref{prop:ex-continuous}) and an economic agent trades on the stock using predictable trading strategies, as described, for example, in the books \cite{DeSc06,KaSh98}. Following the paradigm of \emph{no arbitrage}, the process $X$ must be a semi-martingale by \cite[Theorem 9.7.2]{DeSc06}. This places us in the setting of \cite{DMb} described above, where the usual conditions are in force. An important role is played by the available information that is encoded in the filtration. Consider, for example, the announcement of some economic statistics that is revealed at a fixed time $t$ (corresponding to $t = 0$ in the example of \Cref{prop:ex-continuous}). We claim that the proper interpretation is that the agent may use the knowledge of these statistics from time $t$ on, and not only from time $t + \varepsilon$ for each $\varepsilon > 0$. This point of view corresponds to the right continuity of the filtration. One may ask whether this gives too much information to the agent, since the stock prices $X_t$ may be more favourable than $X_{t + \varepsilon}$ for arbitrary $\varepsilon > 0$. However, this difference in prices is negligible as $\varepsilon \to 0$, since semi-martingales are defined to be right-continuous; see \cite{DMb}. Thus, the usual conditions fit well in this financial setting.
	
	Nevertheless, for purely mathematical interest, one could and should of course ask whether it is possible to construct counterexamples of peacocks for which there exists no mimicking Markov martingale, when the Markov property is defined with respect to the raw filtration. We do not have an answer to this question, and we leave this as an open problem for future research.
\end{remark}

\subsection{The general case}
We now generalize the example given in \Cref{prop:ex-continuous} to find a peacock for which there is no mimicking Markov martingale, even if we allow for jumps. We construct such a peacock by modifying the previous example in the following way. Let us partition the time interval $[0, 1]$ into the intervals $I_1 \coloneqq \bigcup_{\substack{n \in \N\\n\;\text{even}}} [2^{-(n+1)}, 2^{-n}]$, $I_2 \coloneqq \bigcup_{\substack{n \in \N\\n\;\text{odd}}} [2^{-(n+1)}, 2^{-n}]$. Define the functions
	\begin{equation}
		\begin{split}
			a_1(t) & = \int_0^t \mathds{1}_{I_1}(s) \D s, \quad a_2(t) = \int_0^t \mathds{1}_{I_2}(s) \D s.
		\end{split}
	\end{equation}
	Now time-change the processes $M^1, M^2$ from \eqref{eq:r4-simple} to define a process $X$ by
	\begin{equation}\label{eq:r4-dyadic}
		X_t = \begin{cases}
			\!\left(M^1_{a_1(t)}, \; 0\right)\!, \; & \xi = 0,\\
			\!\left(0, \; M^2_{a_2(t)}\right)\!, & \xi = 1,
		\end{cases}\quad t \in [0, 1],
	\end{equation}
	and write $\mu_t = \Law(X_t)$. Then, at time $t \in [0, 1]$, $\mu_t$ is the uniform measure on $S^1_{a_1(t)} \cup S^2_{a_2(t)} \subset \R^4$. Note that, for $t \in I_1$, the radius of $S^1_{a_1(t)}$ is increasing deterministically at rate $\sqrt{t}$, while the radius of $S^2_{a_2(t)}$ remains constant, with the roles reversed on the set of times $I_2$.
	
	We will show that there is no Markov martingale mimicking $\mu$.

\begin{proposition}\label{prop:ex-discontinuous}
	There exists a peacock $\mu$ on $\R^4$ such that there does not exist any \emph{Markov martingale} mimicking $\mu$.
\end{proposition}

\begin{proof}
	Let $\mu_t = \Law(X_t)$, $t \in [0, 1]$, where $X$ is defined by \eqref{eq:r4-dyadic}. Suppose that there exists a martingale $Y$ mimicking $\mu$. As in \Cref{prop:ex-continuous}, we will show that $Y$ is not Markovian at time $0$.
	
	Fix $n \in \N$ even and $t_0 \in [2^{-(n + 1)}, 2^{-n}) \subset I_1$. Then, for all $t \in [2^{-(n+1)}, 2^{-n}]$, $\mu_t$ is supported on $S^1_{a_1(t)} \cup S^2_{a_2(t_0)}$. Define a function $f\colon \R^4 \to \R$ by $f(x) = \sqrt{x_3^2 + x_4^2}$, and note that $f$ is convex. Also note that, for any $t \in [0, 1]$, $f(x) = 0$ for $x \in S^1_{a_1(t)}$, and $f(x) = \sqrt{a_2(t)}$ for $x \in S^2_{a_2(t)}$.
	Let $t \in (t_0, 2^{-n}]$. Then, by convex ordering,
	\begin{equation}\label{eq:cvx-order-ex}
		\begin{split}
			\E\!\left[f(Y_t) \mid Y_{t_0} \in S^2_{a_2(t_0)}\right]\! \geq \E \!\left[f(Y_{t_0}) \mid Y_{t_0} \in S^2_{a_2(t_0)}\right]\! & = \sqrt{a_2(t_0)},\\
			\text{and} \quad\E\!\left[f(Y_t) \mid Y_{t_0} \in S^1_{a_1(t_0)} \right]\! \geq \E \!\left[f(Y_{t_0}) \mid Y_{t_0} \in S^1_{a_1(t_0)} \right]\! & = 0.
		\end{split}
	\end{equation}
	We also have
	\begin{equation}
		\begin{split}
			\E[f(Y_t)] & = \frac12 \sqrt{a_2(t)} = \frac 12 \sqrt{a_2(t_0)} = \E[f(Y_{t_0})].
		\end{split}
	\end{equation}
	Hence equality holds in each inequality in \eqref{eq:cvx-order-ex} and, in particular,
	\begin{equation}
		\P\!\left[Y_t \in S^1_{a_1(t_0)} \mid Y_{t_0} \in S^1_{a_1(t_0)}\right]\! = 1.
	\end{equation}
	This holds for any subinterval of $I_1$, and a symmetric argument applies to $I_2$.
	Since the peacock $\mu$ is weakly continuous, we can suppose that $Y$ has c\`adl\`ag paths. Therefore
	\begin{equation}
		\P \!\left[Y_t \in S^1_{a_1(t)}, \; \forall t \geq t_0  \mid Y_{t_0} \in S^1_{a_1(t_0)} \right]\! = 1.
	\end{equation}
	
	For $t \in (0, 1]$, define $A^1_t \coloneqq \{Y_s \in S^1_{a_1(s)}, \forall s \geq t\}$, and let $A^1 \coloneqq \bigcap_{t > 0}A^1_t$. Then, for each $t \in (0, 1]$, $\P(A^1_t) = 1/2$, and so $\P(A^1) = 1/2$. Since $\sigma(Y_0)$ is trivial, $A^1 \notin \sigma(Y_0)$. On the other hand, we have $A^1_t \in \F^Y_t$. Hence $A^1 \in \F^Y_0 \coloneqq \bigcap_{s > 0}\F^Y_s$.
	
	We conclude as in \Cref{prop:ex-continuous}. Observe that, with $f$ defined as above, for any $t \in (0, 1)$, we get
	\begin{equation}
		\frac{\sqrt{a_2(t)}}{2} = \E[f(Y_t) \mid \sigma(Y_0)] \neq \E[f(Y_t) \mid \F^Y_0] =
		\begin{cases}
			0 & \text{on}\; A^1,\\
			\sqrt{a_2(t)} & \text{on}\; \Omega \setminus A^1.
		\end{cases}
	\end{equation}
	Hence $Y$ is not Markovian at time $0$.
	\end{proof}

\subsection{A partially regularized case}

In this section, we present a final example, in which we regularize a peacock, which is defined similarly as in \Cref{prop:ex-discontinuous}, by convolving the peacock at time $t \in [0, 1]$ with a centered Gaussian with covariance $t \, \id$. We will show that, even after such regularization, there exists no Markov martingale mimicking this peacock. Therefore, in order to guarantee existence of a mimicking Markov martingale, some further regularization is required, as in \Cref{thm:kellerer-regularized}.

Throughout this section, we use the notation $\gamma^\sigma$ for the $4$-dimensional centered Gaussian measure with covariance $\sigma \id$, $\sigma \in \R_+$. Also let $\gamma^0$ denote the Dirac measure at $0 \in \R^4$, so that $\mu_0 \ast \gamma^0 = \mu_0$.

\begin{proposition}\label{prop:ex-regularized}
	There exists a peacock $\mu$ on $\R^4$ such that there is no Markov martingale mimicking $\mu^\reg$, which is defined by $\mu^\reg_t \coloneqq \mu_t \ast \gamma^{t}$, for $t \in [0, 1]$.
\end{proposition}

\begin{proof}
	Let $\nu_t = \Law(X_t)$, where $X_t$ is defined by \eqref{eq:r4-dyadic}, and define a regularized peacock $\nu^\reg$ by $\nu^\reg_t \coloneqq \nu_t \ast \gamma^{t^{14}}$, for $t \in [0, 1]$. Suppose that there exists a martingale $Y$ mimicking $\nu^\reg$. In the following, we will show that $Y$ can not be Markovian at $0$. Finally, we will time-change $\nu$ in order to find a peacock $\mu$ such that any martingale mimicking $\mu^\reg \coloneqq (\mu_t \ast \gamma^t)_{t \in [0, 1]}$ is not Markovian.

	Due to the convolution with a Gaussian $\nu^\reg_t$ is no longer concentrated on $S^1_{a_1(t)} \cup S^2_{a_2(t)}$; it rather has full support on $\R^4$, for all $t \in (0, 1]$. 
    Since $Y$ mimics $\nu^\reg$ we have, for each $t \in [0, 1]$, that $\Law(Y_t) =  \Law( X_t + N_{t^{14}})$ where $X_t \sim \text{Unif}(S^1_{a_1(t)} \cup S^2_{a_2(t)})$ and $N_{t^{14}} \sim \mathcal N(0,t^{14} \id)$ are independent. Note that, from the definitions of $a_1$ and $a_2$, we can find constants $c, C > 0$ such that $c \, t \leq a_i(t) \leq C \, t$, for $t \in [0, 1]$, $i = 1, 2$.
	We also have the estimate
    \begin{equation}
    \begin{split}
        & \P[|N_{t^{14}}| \ge a] \le \frac{t^{14}}{a^2}, \quad a > 0.
	\end{split}
    \end{equation}
	Define the events $\Sccc^1_{t} \coloneqq \{\exists x \in S^i_{a_i(t)} \colon \; |x - Y_t| < t \}$.
	Using the independence of $X_t$ and $N_{t^{14}}$, we have the bound
	\begin{equation}\label{eq:s1-lb}
		\begin{split}
			\P [\Sccc^1_t] & \geq \P[X_t \in S^1_{a_1(t)}, \, |N_{t^{14}}| < t] = \frac 12 \P[|N_{t^{14}}| < t] \geq \frac 12 - \frac{t^{12}}{2}.
		\end{split}
	\end{equation}
	On the other hand, note that
	\begin{equation}
		\{X_t \in S^1_{a_1(t)}\} = \{\exists x \in S^1_{a_1(t)} \colon \; |x - X_t|^2 \leq a_1(t)\}.
	\end{equation}
	Thus
	\begin{equation}\label{eq:s1-ub}
		\begin{split}
			\P [\Sccc^1_t] & \leq t^{14}a_1(t)^{-1} \P [\Sccc^1_t ; \, N_{t^{14}} \geq \sqrt{a_1(t)}] + \P [\Sccc^1_t ; \, N_{t^{14}} \geq \sqrt{a_1(t)}]\\
			& \leq c^{-1}t^{13} + \P[ \exists x \in S^1_{a_1(t)} \colon \; |x - X_t|^2 \leq a_1(t)] = c^{-1}t^{13} + \frac 12.
		\end{split}
	\end{equation}
	
	Write $t_k \coloneqq 2^{-k}$ for $k \in \N$. By \Cref{lem:app-reg} we have that,
	for $t \in [0, 1]$,
    \begin{equation}
        \label{eq:ex_regularized.difference}
        \P[\Sccc^1_{t_{k}} \bigtriangleup \Sccc^1_{t_{k + 1}}] \le C t_k,
    \end{equation}    
    where $A \bigtriangleup B \coloneqq (A \setminus B) \cup (B \setminus A)$ denotes the symmetric difference between events $A$ and $B$.
    Therefore, for $m,n \in \N$, $m \ge n$, the bounds \eqref{eq:s1-lb} and \eqref{eq:ex_regularized.difference} imply
    \begin{equation}
    \begin{split}
        \P\!\left[\bigcap_{k = n}^{m} \Sccc_{t_k}^1 \right]\! & = \P\!\left[\Sccc_{t_m}^1 \setminus \bigcup_{k = n}^{m - 1} (\Sccc_{t_k}^1 \bigtriangleup \Sccc_{t_{k + 1}}^1)\right]\! \geq \P[\Sccc_{t_m}^1] - \sum_{k = n}^{m-1}\P[\Sccc_{t_k}^1 \bigtriangleup \Sccc_{t_{k + 1}}^1]\\
        & \geq \frac12 - \frac{t_m^{12}}{2} - C\sum_{k = n}^{m - 1}t_k\ge \frac12 - C \sum_{k = n}^{m} t_k \xrightarrow{m, n \to \infty} \frac12,
    \end{split}
    \end{equation}
    since the sequence $(t_k)_{k \in \N}$ is summable, and \eqref{eq:s1-ub} gives
    \begin{equation}
    	\P\!\left[\bigcap_{k = n}^{m} \Sccc_{t_k}^1 \right]\! \leq \inf_{n \leq k \leq m} \P[\Sccc^1_{t_k}] \leq \frac12 + c^{-1} t_m^{13} \xrightarrow{m \to \infty} \frac12.
    \end{equation}
    Defining an increasing sequence of events $\Sc_n$, $n \in \N$, and its limit $\Sc$ by
    \begin{equation}
        \Sc_n \coloneqq \bigcap_{k = n}^\infty \Sccc_{t_k}^1, \quad \text{and} \quad \Sc \coloneqq \bigcup_{n = 1}^\infty \Sc_n,
    \end{equation}
    we conclude that
    \begin{equation}
    	\P[\Sc] = \lim_{n \to \infty} \P[\Sc_n] = \frac{1}{2}.
    \end{equation}
    Hence $\Sc \notin \sigma(Y_0)$. However, we have that, for each $k \in \N$, $\Sccc^1_{t_k} \in \F^Y_{t_k}$, and so for $n \in \N$, $\Sc_n \in \F^Y_{t_k}$ for all $k \geq n$. Hence $\Sc \in \bigcap_{k \geq n} \F^Y_{t_k} = \F^Y_0$.
    
    Now choose $k \in \N$ sufficiently large that $\P[\Sc \setminus \Sc_k] \le 1/8$ and $t_k \leq 1/8$.
    Then we have
    \begin{equation}
    	\P[\Sccc^1_{t_k} \mid \Sc] = \P[\Sccc^1_{t_k} \cap \Sc]/\P[\Sc] \geq  2 \P[\Sc_k \cap \Sc] = 2(\P[\Sc] - \P[\Sc \setminus \Sc_k]) \geq 2\!\left(\frac12 - \frac18\right)\! = \frac34,
    \end{equation}
    while on the other hand, $\P[\Sccc^1_{t_k} \mid \sigma(Y_0)] \leq 1/2 + t_k \leq 5/8 < 3/4$. Therefore $Y$ cannot be Markovian at time $0$.
    
    Now define a peacock $\mu$ by a time-change of $\nu$ such that $\mu_t \coloneqq \nu_{t^{14^{-1}}}$, and define a regularized peacock $\mu^\reg$ by $\mu^\reg_t \coloneqq \mu_t \ast \gamma^t$, $t \in [0, 1]$. Then, rescaling time by $t \mapsto t^{14^{-1}}$ in all of the above arguments, we obtain the result that any martingale mimicking $\mu^\reg$ cannot be Markovian at time $0$.
\end{proof}

\section{Compactness of martingale It\^o diffusions}\label{sec:compactness-Ito}

In this section we prove a compactness result for martingale diffusions with respect to convergence in finite dimensional distributions (Theorem \ref{thm:fdd_convergence}).
We applied this result in the proof of \Cref{thm:kellerer-regularized} in order to pass to a limit when constructing a mimicking martingale diffusion. This parallels the approach of Lowther \cite{Lo09} to the one-dimensional case.

Consider a sequence $(\sigma^k)_{k \in \N}$ of positive semidefinite matrix-valued measurable functions $\sigma^k \colon [0,1] \times \R^d \to \R^{d \times d}$.
Let $\Sigma^k \colon [0,1] \times \R^d \to \R^{d \times d}$ denote the integral
\begin{equation}
    \Sigma^k_t(x) \coloneqq \int_0^t \sigma_s^k(x)^2 \, \D s.
\end{equation}
Moreover, fix a sequence of initial distributions $(\mu_0^k)_{k \in \N}$ and suppose that there exist weak solutions $(X^k)_{k \in \N}$ of the SDEs
\begin{equation}\label{eq:sde-sequence}
    dX_t^k = \sigma^k_t(X_t^k) \, \D B^k_t, \quad \mbox{ with }X_0^k \sim \mu_0^k,
\end{equation}
where $B^k$ denotes a standard $\R^d$-valued Brownian motion, and $(X^k, B^k)$ is defined on some probability space $(\Omega^k, \F^k, \P^k)$, for each $k \in \N$.
Let $(\mu^k_t)_{t \in [0,1]}$ denote the marginal distributions of $X^k$, for each $k \in \N$.

In the following we will use combinations of the following assumptions, which were satisfied in the setting of \Cref{prop:pasting-intervals}.

\begin{assumption}\label{ass:cpct}\;

	\begin{enumerate}[label = (A\arabic*)]
	    \item \label[ass]{it:assumption_1}
	    The map $x \mapsto \sigma^k_t(x)^2$ is locally Lipschitz continuous, uniformly in $k \in \N$ and $t \in [0,1]$.
	    \item \label[ass]{it:assumption_2}
	    For every $x \in \R^d$ the value of $\| \sigma^k_t(x) \|$ is bounded, uniformly in $k \in \N$ and $t \in [0,1]$.
	    \item \label[ass]{it:assumption_3}
	    The family of random variables $\{ |X^k_1|^2 \colon k \in \N \}$ is uniformly integrable.
	    \item \label[ass]{it:assumption_4}
	    The matrix $\sigma^k_t(x)^2$ is positive definite with eigenvalues bounded away from zero, locally in $x \in \R^d$, uniformly in $t \in \bigcup_{j = 0}^{2^k - 1} [j 2^{-k}, j 2^{-k} +2^{-k-1}],$ and $k \in \N$.
	    \item \label[ass]{it:assumption_5}
	    The set of initial distributions $\{ \mu_0^k \colon k \in \N \}$ converges to $\mu_0 \in \mathcal P_2(\R^d)$.
	\end{enumerate}
\end{assumption}

\begin{remark}
	Due to \Cref{it:assumption_1,it:assumption_2} and the Arzel\`a-Ascoli theorem, we can assume without loss of generality, by passing to subsequences, that $(\Sigma^k_t(x))_{k \in \N}$ converges for every $(t, x) \in [0, 1] \times \R^d$.
\end{remark}
	
\begin{remark}[continuity of matrix square root]\label{rem:square-root}
	Suppose that, for some domain $\mathcal O \subseteq \R^d$, and some function $\theta \colon \mathcal O \to \R^{d \times d}$, the square $\theta^2 \colon \mathcal O \to \R^{d \times d}$ is Lipschitz continuous. Then
	\begin{enumerate}[label = (\roman*)]
		\item $x \mapsto \theta(x)$ is $1/2$-H\"older  continuous on $\mathcal O$, since the matrix square root is $1/2$-H\"older  continuous on the set of positive semidefinite matrices \cite[Theorem 1.1]{Wi09}.
	\end{enumerate}
	If, moreover, the eigenvalues of $\theta(x)^2$ are bounded away from zero for each $x \in \mathcal O$, then
	\begin{enumerate}[label = (\roman*), start = 2]
		\item $x \mapsto \theta(x)$ is Lipschitz on $\mathcal O$, since the matrix square root is Lipschitz on the set of positive definite matrices with eigenvalues bounded away from zero.
	\end{enumerate}
	
	In particular, \Cref{ass:cpct} implies that $x \mapsto \sigma^k_t(x)$ is locally $1/2$-H\"older continuous uniformly in $t \in [0, 1]$ and $k \in \N$, and locally Lipschitz continuous uniformly in $t \in \bigcup_{j = 0}^{2^k - 1} [j 2^{-k}, j 2^{-k} +2^{-k-1}]$ and $k \in \N$.
	Thus the SDEs \eqref{eq:sde-sequence} may not admit unique strong solutions; see \cite[Remark 2]{YW2}.
\end{remark}

\begin{theorem}
    \label{thm:fdd_convergence}
    Suppose that there exist weak solutions $(X^k)_{k \in \N}$ of \eqref{eq:sde-sequence}, \Crefrange{it:assumption_1}{it:assumption_5} are satisfied, and $(\Sigma^k_t(x))_{k \in \N}$ converges pointwise for $(t,x) \in [0,1] \times \R^d$ to $\Sigma \colon [0,1] \times \R^d \to \R^{d \times d}$.
    Then there exists a function $(t, x) \mapsto \sigma_t(x)$ taking values in the set of positive definite $d \times d$-matrices such that, uniformly in $t \in [0, 1]$,
    \begin{enumerate}[label = (\roman*)]
    	\item $(t, x) \mapsto \sigma_t(x)^2$ is locally Lipschitz continuous,
    	\item for each $x \in \R^d$, there exist constants $c, C > 0$ such that $c \, \id \leq \sigma_t(x)^2 \leq C \, \id$.
    \end{enumerate}
    Moreover, let $(\Omega, \F, \P)$ be any probability space supporting a standard $\R^d$-valued Brownian motion $B$ and independent random variable $\xi \sim \mu_0$. Then there exists a unique strong solution $X$ of the SDE $dX_t = \sigma_t(X_t) \, \D B_t$, with $X_0 \sim \mu_0$, and $(X^k)_{k \in \N}$ converges in finite dimensional distributions to $X$.
\end{theorem}

As a simple corollary we have the following compactness result.

\begin{corollary}
    Under \Crefrange{it:assumption_1}{it:assumption_5}, the set of martingale It\^o diffusions $\{ X^k \colon k \in \N \}$ is precompact with respect to convergence in finite dimensional distributions in the set of martingale It\^o diffusions.
\end{corollary} 

We start with two auxiliary lemmas.

\begin{lemma}
    \label{lem:marginal_curves}
    Suppose that there exist weak solutions $(X^k)_{k \in \N}$ of \eqref{eq:sde-sequence}.
    Under \Crefrange{it:assumption_1}{it:assumption_3}, the sequence of curves $t \mapsto \mu^k_t$, $t \in [0,1]$, of marginal distributions of $(X^k)_{k \in \N}$ is equicontinuous in $C([0,1], \mathcal P_2(\R^d))$ with respect to the $\W_2$-metric on $\mathcal P_2(\R^d)$.
\end{lemma}

\begin{proof}
       Since by \Cref{it:assumption_3} the set of terminal distributions is $\W_2$-precompact, the set $\{\eta \in \mathcal P_2(\R^d) \colon \exists k \in \N \mbox{ with }\eta \preceq  \mu_1^k \}$ is also $\W_2$-precompact.
       Applying Doob's maximal $L^2$-inequality, for each $\varepsilon > 0$, we can find a ball $B_R \subseteq \R^d$ of radius $R > 0$ such that for all $t \in [0,1]$ and $k \in \N$, $\W_2(\mu_t^k,\mu_t^{k,R}) < \varepsilon$ where
    \begin{equation}
        X^{k,R}_t \coloneqq X^k_{\tau^k \wedge t}, \quad \mu^{k,R}_t \coloneqq \operatorname{Law}(X^{k,R}_t), \quad
        \tau^k \coloneqq \inf \{ s > 0 \colon X^k_s \notin B_R \}.
    \end{equation}
    Next, we show that the curves $(\mu^{k,R})_{k \in \N}$ are $1/2$-H\"older continuous with uniform H\"older constant $\sqrt{\Lambda_\varepsilon} > 0$.
    Indeed, by the It\^o isometry,    \Cref{it:assumption_1,it:assumption_2}, we get, for $0 \le t_0 \le t_1 \le 1$, $k \in \N$,
    \begin{align}
        \mathcal W_2^2(\mu_{t_0}^{k,R}, \mu_{t_1}^{k,R}) &\le \mathbb E \!\left[| X_{\tau^k \wedge t_1}^k - X_{\tau^k \wedge t_0}^k|^2 \right]\! = \mathbb E\!\left[ \int_{t_0 \wedge \tau^k }^{t_1 \wedge \tau^k} \|\sigma^k_s(X_s^k)\|^2 \, \D s\right]\!
        \\
        &\le 
        3 \mathbb E \!\left[ \int_{t_0 \wedge \tau^k}^{t_1 \wedge \tau^k} \!\left(\|\sigma_s^k(0)\|^2 + \|\sigma_s^k(0) - \sigma_s^k(X_t^k)\|^2 + \|\sigma_s^k(X_t^k) - \sigma_s^k(X_s^k)\|^2\right)\! \D s \right]\!
        \\
        &\le
        3 (t_1 - t_0) \!\left( C + L_R^2 \mathbb E\!\left[|X_{t_1}^k|^2\right]\! + L_R^2 \mathbb E\!\left[|X_{t_0}^k - X_{t_1}^k|^2\right]\! \right)\!,
    \end{align}
    where $L_R$ denotes the Lipschitz constant for $\sigma^k_s$ on $B_R$ that is provided by \Cref{it:assumption_1}.
    Setting $\Lambda_\varepsilon \coloneqq 3( C + 2 L_R^2 \sup_{k \in \N} \mathbb E [|X_1^{k}|^2])$ we obtain a uniform $1/2$-H\"older bound. We conclude by noting that uniform convergence preserves equicontinuity.
	In fact the function $\delta \mapsto \inf_{\varepsilon > 0} \{ 2 \varepsilon + \sqrt{ \Lambda_\varepsilon \delta } \}$ is a modulus of continuity for the sequence $(\mu^k)_{k \in \N}$.
  \end{proof}

For the following lemma, compare to Beiglb\"ock, Huesmann, Stebegg \cite[Theorem 1]{BeHuSt16}. Let $\mathcal D([0, 1]; \R^d)$ denote the Skorohod space of $\R^d$-valued c\`adl\`ag paths over the time interval $[0, 1]$.

\begin{lemma}\label{lem:cpct-continuity-points}
    Let $\Lambda$ be a $\W_2$-compact subset of $\mathcal P_2(\R^d)$ and let $\mathcal M(\Lambda)$ denote the set of probability measures on $\mathcal D([0, 1]; \R^d)$ defined by
    \begin{equation} 
        \mathcal M(\Lambda) \coloneqq \{ \pi = \Law(M) \colon (M_t)_{t \in [0, 1]} \text{ is a c\`adl\`ag martingale with }\Law(M_1) \in \Lambda \}.
    \end{equation}
    Then, for any sequence $(\pi^k)_{k \in \N}$ in $\mathcal M(\Lambda)$ there exists $\pi \in \mathcal M(\Lambda)$ and a subsequence $(\pi^{k_j})_{j \in \N}$ that converges to $\pi = \Law(M)$ in finite dimensional distributions on the set of continuity points with respect to the weak topology of the function $t \mapsto \Law(M_t)$.
\end{lemma}

\begin{proof}
    Let $(\pi^k)_{k \in \N}$ be a sequence in $\mathcal M(\Lambda)$ and write $M^k$ for a martingale with law $\pi^k$.
    Since $\Lambda$ is compact, it follows
    that $\{ \mu \in \mathcal P_2(\R^d) \colon \exists \nu \in \Lambda \mbox{ with }\mu \le_c \nu \}$ is also compact.
    Therefore we find a subsequence $(\pi^{k_j})_{j \in \N}$ such that, for any finite subset $S \subseteq [0,1] \cap \mathbb Q$,
    \begin{equation}
        \Law(M_t^{k_j})_{t \in S} \to \tilde \pi^S \mbox{ weakly for }j\to\infty,    
    \end{equation}
    where $\tilde \pi^S$ is the law of a discrete-time martingale in $|S|$ time steps with values in $\R^d$.
    The family $(\tilde \pi^S)_{S \subset [0,1] \cap \mathbb Q, |S| < \infty}$ is a consistent family and we can apply Kolmogorov's extension theorem to obtain a probability $\tilde \pi$ on $\prod_{t \in [0,1] \cap \mathbb Q} \mathbb R^d$.
    Note that, for any $S \subseteq [0,1] \cap \mathbb Q$, the projection of $\tilde \pi$ onto the $S$-coordinates coincides with $\tilde \pi^S$.
    Hence $\tilde \pi$ is the law of a martingale $\tilde M = (\tilde M_t)_{t \in [0,1] \cap \Q}$ with terminal distribution $\Law(\tilde M_1) \in \Lambda$.
    By standard arguments, there exists $M$ where $M_t \coloneqq \lim_{q \searrow t, \, q \in \Q \cap [0,1]} \tilde M_q$ for $t \in [0,1]$ which is a c\`adl\`ag martingale (in the right-continuous version of the filtration).
    We claim that $\pi \coloneqq \Law(M)$ has the desired properties.

    As $t \mapsto \textrm{Var}(M_t)$ is non-decreasing there are at most countably many points of discontinuity.
    Let $U$ be a finite subset of the continuity points of $t \mapsto \textrm{Var}(M_t)$, which coincide with the continuity points of $t \mapsto \Law(M_t)$.
    Fix $N \in \N$ and note that, as all involved processes are martingales, $(M^{k_j})_{t \in \tilde S}$ converges for $j \to \infty$ in $\mathcal W_2$ to $(\tilde M_t)_{t \in \tilde S}$ uniformly for all $\tilde S \subseteq [0,1] \cap \Q$ with $|\tilde S|\le N$.
    Moreover, by Doob's martingale convergence theorem, we have, for any $t \in U$, that $\lim_{q \searrow t, \, \in [0,1] \cap \Q} \tilde M_q = M_t$ almost surely.
    We conclude that $(M^{k_j}_t)_{t \in U}$ converges in $\mathcal W_2$ to $(M_t)_{t \in U}$.
\end{proof}

\begin{proposition}
    \label{prop:fdd_convergence}
    Under \Crefrange{it:assumption_1}{it:assumption_5}, suppose moverover that there exists $R > 0$ such that $(\sigma^k)_{k \in \N}$ satisfies $\sigma^k_t(x) = \sigma^k_t(\frac{R x}{|x| \vee R})$ for $t \in [0, 1]$, $x \in \R^d$, $k \in \N$, and that $(\Sigma^k_t(x))_{k \in \N}$ converges pointwise for $(t,x) \in [0,1] \times \R^d$ to $\Sigma \colon [0,1] \times \R^d \to \R^{d \times d}$.
	
    Then the conclusion of \Cref{thm:fdd_convergence} holds.
\end{proposition}

We break the proof of \Cref{prop:fdd_convergence} into the following lemmas.

\begin{lemma}
    \label{lem:Lipschitz_continuity_sigma}
    In the setting of \Cref{prop:fdd_convergence}, there exists $\sigma \colon [0,1] \times \R^d \to \R^{d \times d}$ with $\Sigma_t(x) = \int_0^t\sigma_s(x)^2 \D s$ such that, uniformly in $t \in [0, 1]$, $x \mapsto \sigma_t(x)^2$ is locally Lipschitz continuous and, for each $x \in \R^d$, there exist constants $c_x, C_x > 0$ sarisfying $c_x \, \id \leq \sigma_t(x)^2 \leq C_x \, \id$.
\end{lemma}

\begin{proof}
	The specific form of the $\sigma^k$ allows us to restrict to the ball $B_R \coloneqq \{ x \in \R^d \colon |x| < R \}$ of radius $R > 0$.
    As the limit of Lipschitz functions, $t \mapsto \Sigma_t(x)$ is Lipschitz continuous, and so are the entries $(\Sigma^{i,j})_{i,j = 1}^d$ of $\Sigma$.
    Therefore there exist densities
    \begin{equation}
        \rho_t(x) \coloneqq (\rho^{i,j}_t(x))_{i,j = 1}^d \in \R^{d \times d},  \mbox{ where } \int_{t_0}^{t_1} \rho_t(x) \, \D t = \Sigma_{t_1}(x) - \Sigma_{t_0}(x), \quad x \in \R^d.
    \end{equation}
    We define $\sigma$ as the matrix square root of $\rho$, which is possible as $\rho$ is a.s.\ positive semidefinite.
    Next, we define the Lipschitz norm of a function $g \colon [0,1] \times B_R \to \R^{d \times d}$ as
    \begin{align}
        F(g)  \coloneqq& \operatorname{esssup}_{x,y \in B_R, t \in [0,1]} \frac{\|g_t(x) - g_t(y)\|}{|x - y|},
    \end{align}
    where the essential supremum is taken with respect to $\D t \otimes \D x$.
    Since $F \colon L^2([0,1] \times B_R; \R^{d \times d}) \to \R_+ \cup \{ \infty\}$ is lower semicontinuous and convex, we have by \cite[Theorem 9.1]{BaCo11} that $F$ is weakly lower semicontinuous, and in particular
    \begin{equation}
        \liminf_{j \to \infty} F((\sigma^{k_j})^2) \ge F(\rho) \eqqcolon L,
    \end{equation}
    which implies that $x \mapsto \rho_t(x)$ is $dt \otimes dx$-almost everywhere $L$-Lipschitz continuous.
    Thus, by choosing a suitable $L^2$-representative of $\rho$ we can assume without loss of generality that $x \mapsto \rho_t(x)$ is $L$-Lipschitz continuous in $x$ for every $t \in [0, 1]$, and that $\sup_{t \in [0,1]} \|\rho_t(x)\| < \infty$.
    By Lipschitz continuity of $t \mapsto \Sigma^k_t(x)$ and $x \mapsto \rho_t(x)$, we have that, for every $x \in B_R$, $\xi \in \R^d$ and $0 \leq t_0 \leq t_1 \leq 1$,
    \begin{equation}
        \int_{t_0}^{t_1} \xi^\textrm{T} \rho_t(x) \xi \, \D t = \lim_{k \to \infty} \xi^\textrm{T} \!\left( \Sigma^k_{t_1}(x) - \Sigma^k_{t_0}(x) \right)\! \xi.
    \end{equation}
    Therefore, by \Cref{it:assumption_4}, there exists a constant $c > 0$ such that $\rho_t(x) \ge c \, \id$ for every $x \in B_R$ and Lebesgue-almost every $t \in [0,1]$.
 \end{proof}

The second result that we will make use of in the proof of \Cref{prop:fdd_convergence} is the Lipschitz continuity of the $\W_2$-distance between Gaussian laws with respect to covariance matrices. From now on, we let $\mathcal N(\mu, \sigma^2)$ denote the law of a normal random variable with mean $\mu \in \R^d$ and covariance matrix $\sigma^2 \in \R^{d \times d}$.

\begin{lemma}
    \label{lem:Lipschitz_continuity_gaussians}
    Let $C, \delta > 0$.
    Then the map
    \begin{equation}
        \R^{d \times d} \times \R^{d \times d} \ni (\sigma^2, (\sigma')^2) \mapsto \mathcal W_2\!\left( \mathcal N(0,\sigma^2), \mathcal N(0,(\sigma')^2)\right)\!    
    \end{equation}
    is Lipschitz continuous on the set $\{(\sigma, \sigma^\prime) \colon \; \delta \, \id \leq \sigma^2 \leq C \, \id, \; \delta \, \id \leq (\sigma^\prime)^2 \leq C \, \id \}$ equipped with the product of the Hilbert--Schmidt norm.
\end{lemma}

\begin{proof}
    As shown in \cite[Proposition 7]{GiSh84}, the squared Wasserstein-2 distance between two centered Gaussians with covariance matrices $\sigma, \sigma^\prime$ is explicitly given by
    \begin{equation}
        \label{eq:Wasserstein_2_gaussians}        
        \trace \!\left( \sigma^2 + (\sigma')^2 - 2\!\left( \sigma (\sigma')^2 \sigma \right)\!^\frac12 \right)\!,
    \end{equation}
    from which the assertion follows, in light of \Cref{rem:square-root}.
\end{proof}

\begin{proof}[Proof of \Cref{prop:fdd_convergence}]
	For each $k \in \N$ and $t \in [0, 1]$, the functions $x \mapsto \sigma^k_t(x)$ are bounded and continuous. Thus, by \cite[Theorem 6.1.7]{StVa79}, there exists a weak solution $X^k$ of \eqref{eq:sde-sequence} for all $k \in \N$.
    By Lemma \ref{lem:marginal_curves} there exists a subsequence, still denoted by $(X^k)_{k \in \N}$, such that the curves $(\mu^k_t)_{t \in [0,1]}$, $k \in \N$, converge to a $\mathcal W_2$-continuous curve $(\mu_t)_{t \in [0,1]}$.
    After a deterministic time-change, if necessary, we can assume without loss of generality that $t \mapsto \int |x|^2 \, \mu_t(dx)$ is 1-Lipschitz.
    
    By Lemma \ref{lem:Lipschitz_continuity_sigma}, we can find a diffusion coefficient $\sigma$ with the desired properties. Then by \Cref{rem:square-root}, the map $x \mapsto \sigma_t(x)$ is Lipschitz continuous uniformly in $t \in [0,1]$, with some Lipschitz constant $\tilde L$. Thus there exists a unique strong solution $X$ of $\D X_t = \sigma(X_t) \D B_t$ with $X_0 \sim \mu_0$.
    The particular form of the $(\sigma^k)_{k \in \N}$ now gives us the following properties. \Cref{it:assumption_1} implies that $x \mapsto \sigma^k_t(x)^2$ is globally Lipschitz continuous uniformly in $t \in [0, 1]$ and $k \in \N$. Thus, by \Cref{rem:square-root}, $x \mapsto \sigma^k_t(x)$ is globally $1/2$-H\"older continuous uniformly in $t \in [0,1]$ and $k \in \N$, with some H\"older constant $L$.
    \Cref{it:assumption_2} implies that there exists a constant $C \geq 0$ such that $\|\sigma_t^k(x)\| \leq C$ for all $t \in [0, 1]$, $x \in \R^d$, $k \in \N$.
    
    Now, for $m \in \N$ and $k = 0,\ldots,2^m$, define $t_k^m \coloneqq k 2^{-m}$ and, for $n \in \N$, consider the kernels
\begin{align}
    \pi_{m,k}^{n}(x) \coloneqq \Law( X^{n}_{t^m_{k+1}} \mid X^{n}_{t^m_k} = x),
    \qquad \bar \pi_{m,k}^{n}(x) \coloneqq \Law( Y^{n}_{t^m_{k+1}} \mid Y^{n}_{t^m_k} = x )
\end{align}
where $Y^{n}_{t^m_k} = x$ and $Y^{n}_{t^n_{k + 1}} = \int_{t^m_k}^{t^{m}_{k + 1}}\sigma_t^{n}(x) \, \D B^n_t$, for $x \in \R^d$.
Observe that, for each $n \in \N$, $x \in \R^d$,
\begin{equation}
    \label{eq:bar_pi_normal_distribution}
    \bar \pi^{n}_{m,k}(x) = \mathcal N\!\left(x, \int_{t^m_k}^{t^m_{k+1}} \sigma_t^{n}(x)^2 \, \D t \right)\! = \mathcal N\!\left(x, \Sigma_{t^m_{k+1}}^{n}(x) - \Sigma_{t^m_k}^{n}(x) \right)\!. 
\end{equation}
In a similar manner we define $Y$, $\pi_{m,k}$ and $\bar \pi_{m,k}$.

Note that $(\Sigma^n)_{n \in \N}$ converges uniformly on $B_R$, and therefore on $\R^d$, to $\Sigma$. Thus \eqref{eq:bar_pi_normal_distribution} and Lemma \ref{lem:Lipschitz_continuity_gaussians} imply the weak convergence
\begin{equation}
    \label{eq:bar_pi_convergence}
    \lim_{n \to \infty}\bar \pi^n_{m,k}(x) = \bar \pi_{m,k}(x)\quad \text{uniformly in }x \in \R^d, \mbox{ uniformly in }k \mbox{ and }m.
\end{equation}
Combining this with the bound $\|\Sigma_t\| \le C^2$ gives
\begin{equation}
    \label{eq:prop_bar_pi_convergence_expectation}
    \lim_{n \to \infty}
    \sum_{k = 0}^{2^m - 1} \mathbb E \!\left[ \mathcal W_2^2\!\left(\bar \pi^n_{m,k}(X^n_{t^m_k}), \bar \pi_{m,k}(X^n_{t^m_k})\right)\! \right]\! = 0.
\end{equation}
In the following we choose $n \coloneqq n(m) \in \N, n(m) \ge m$ sufficiently large such that for this particular $n$ the sum in \eqref{eq:prop_bar_pi_convergence_expectation} is smaller than $2^{-m}$.
We estimate
\begin{align}
    \nonumber
    \mathcal W_2^2\!\left( \pi^n_{m,k}(x), \bar \pi^n_{m,k}(x) \right)\! 
    &\le 
    \mathbb E \!\left[ |X^n_{t^m_{k+1}} - Y^n_{t^m_{k+1}} |^2 \mid X^n_{t^m_k} = x = Y^n_{t^m_k} \right]\!
    \\
    \nonumber
    &=
    \mathbb E \!\left[ \int_{t^m_k}^{t^m_{k+1}} \|\sigma^n_s(X^n_s) - \sigma^n_s(x) \|^2 \, \D s \mid X^n_{t^m_k} = x \right]\!
    \\
    \nonumber
    &\le
    L \mathbb E \!\left[ \int_{t^m_k}^{t^m_{k+1}} | X^n_s - X^n_{t^m_k} | \, \D s \mid X^n_{t^m_k} = x \right]\!,
    \end{align}
using the It\^o isometry and the H\"older continuity of $\sigma^n$. Now, since $X^n$ is a square-integrable martingale, we have
    \begin{align}
    \mathcal W_2^2\!\left( \pi^n_{m,k}(x), \bar \pi^n_{m,k}(x) \right)\! 
    &\le
    L \mathbb E \!\left[ \int_{t^m_k}^{t^m_{k+1}} | X^n_{t^m_{k+1}} - X^n_{t^m_k} | \, \D s \mid X^n_{t^m_k} = x \right]\!
    \\
    \nonumber
    &=
    L 2^{-m} \mathbb E \!\left[ |X^n_{t^m_{k+1}} - X^n_{t^m_k}| \mid  X^n_{t^m_k} = x \right]\!
    \\
    \nonumber
    &=
    L 2^{-m} \mathbb E \!\left[  \!\left(\int_{t^m_k}^{t^m_{k+1}} \| \sigma^n_s(X^n_s) \|^2 \, \D s\right)\!^\frac12 \mid X^n_{t^m_k} = x  \right]\!
    \\
    \label{eq:prop_pi_n_m_convergence}
    &\le
    L C 2^{-\frac{3m}{2}},
\end{align}
where we use the It\^o isometry again, as well as the bound on the norm of $\sigma^n$. By the same line of reasoning, and using the Lipschitz property of $\sigma$, we find that
\begin{equation}\label{eq:wass-2-bound-sigma}
	\mathcal W_2^2(\pi_{m,k}(x), \bar \pi_{m,k}(x)) \leq \tilde L^2 C^2 2^{-2m}.
\end{equation}
Hence, for some constant $\tilde C > 0$, the triangle inequality together with the above estimates yields
\begin{equation}
    \label{eq:prop_sum_W2_vanishes}
    \sum_{k = 0}^{2^m - 1} \mathbb E \!\left[ \mathcal W_2^2\!\left(\pi^{n(m)}_{m,k}(X^{n(m)}_{t^m_k}), \pi_{m,k}(X^{n(m)}_{t^m_k})\right)\! \right]\! 
    \le
    \tilde C 2^{-\frac{m}{2}}.
\end{equation}

Now fix a probability space $(\Omega, \F, \P)$ supporting a standard $\R^d$-valued Brownian motion $B$ and independent random variables $\xi_0 \sim \mu_0$ and $\xi_k \sim \mu^k_0$, $k \in \N$.
For each $m \in \N$, we define an auxiliary process $S^m$ that has c\`adl\`ag paths and the same marginals as $X^{n(m)}$ at the $m$-dyadics, where $n(m)$ is fixed after \eqref{eq:prop_bar_pi_convergence_expectation}.
Set $S_0^m = X_0^{n(m)}$, define $S^m$ as the unique strong solution of $\D S^m_t = \sigma_t(S^m_t) \D B_t$ on the interval $[0, t^m_1)$ and, on each interval $[t^m_k, t^m_{k + 1})$, $k = 1,\ldots,2^m-1$, the unique strong solution of 
\begin{equation}
    \label{eq:prop_S_m_definition}
   \D S_t^m = \sigma_t(S_t^m) \, \D B_t, \quad S_{t^m_{k}}^m = T^m_k(S^m_{t^m_{k}-}),
\end{equation}
where $T^m_k$ is the $\mathcal W_2$-optimal map between $\pi_{m, k - 1}(S_{t^m_{k-1}}^m)$ and $\pi^n_{m,k - 1}(S_{t^m_{k-1}}^m)$.
The discrete-time jump process $Z^{m}_{t} \coloneqq \sum_{l = 1}^{2^m} \mathds 1_{[0,t]}(t^m_{l})(S_{t^m_{l}}^{m} - S_{t^{m}_{l}-}^{m})$ is a martingale in the underlying filtration and $\tilde S^m \coloneqq S^{m} - Z^{m}$ is a continuous martingale.
Indeed
\begin{align}
    \mathbb E \!\left[ Z^{m}_{\hat t} \mid \mathcal F_t \right]\! 
    &= 
    Z_t^{m} + \sum_{l = 1}^{2^m} \mathds 1_{(t,\hat t]}(t^m_{l})\mathbb E \!\left[ S_{t^m_{l}}^{m} - S_{t^m_{l}-}^{m} \mid \mathcal F_t \right]\! 
    \\
    &= Z_t^{m} +  \sum_{l = 1}^{2^m} \mathds 1_{(t,\hat t]}(t^m_{l}) \mathbb E \!\left[ S_{t^m_{l-1}}^{m} - S_{t^m_{l-1}}^{m} \mid \mathcal F_t \right]\! = Z_t^{m}.
\end{align}
Moreover, by \eqref{eq:prop_sum_W2_vanishes}, $Z^{m}$ admits the following estimate
\begin{align}
    \mathbb E \!\left[ | Z^{m}_1 |^2 \right]\!^\frac12 
    = \mathbb E \!\left[ \sum_{k = 0}^{2^m-1} \mathcal W_2^2\!\left(\pi_{m,k}(X^{n(m)}_{t^m_k}),\pi^{n}_{m,k}(X_{t^m_k}^{n(m)})\right)\! \right]\!^\frac12 
    \le
    \tilde C^\frac12 2^{- \frac{m}{4}},
\end{align}
whence, by Doob's maximal inequality, the term $\mathbb E [\sup_{t \in [0,1]} |Z^m_t|]$ also vanishes as $m \to \infty$.
We claim that $(\tilde S^{m})_{m \in \N}$ is a Cauchy sequence.
Indeed, using the Lipschitz property of $\sigma$ and the bound \eqref{eq:wass-2-bound-sigma}, for $m, \hat m \in \N$, $\hat m \ge m$ we have
\begin{align}
    \mathbb E \!\left[ |\tilde S^{m}_t - \tilde S^{\hat m}_t|^2 \right]\!
    &\le
    2
    \!\left( 
        \mathbb E\!\left[ |X^{n(m)}_0 - X^{n(\hat m)}_0|^2 \right]\! +
        \mathbb E\!\left[ \int_0^t |\sigma_t(S^{m}_{\hat t}) - \sigma_t(S^{\hat m}_{\hat t})|^2 \, \D \hat t\right]\!
    \right)\!
    \\
    &\le
    2
    \!\left( 
        \mathbb E\!\left[ |X^{n(m)}_0 - X^{n(\hat m)}_0|^2 \right]\! + \tilde L^2
        \mathbb E\!\left[ \int_0^t |S^{m}_{\hat t} - S^{\hat m}_{\hat t}|^2 \, \D \hat t\right]\!
    \right)\!
    \\
    &\le
    2
    \!\left( 
        \mathbb E\!\left[ |X^{n(m)}_0 - X^{n(\hat m)}_0|^2 \right]\! +
        3 \tilde L^2 
        \!\left(
            \mathbb E \!\left[ \int_0^t |\tilde S^{m}_{\hat t} - \tilde S^{\hat m}_{\hat t}|^2 \, \D \hat t\right]\!
            + \tilde L^2 C^2 2^{1 - 2m}
        \right)\!
    \right)\!.
\end{align}
By Gr\"onwall's lemma, we have that $(\tilde S^{m}_1)_{m \in \N}$ is an $L^2$-Cauchy sequence.
Therefore, there exists a continuous $L^2$-martingale $S$ such that $(\tilde S^{m}_1)_{m \in \N}$ converges in $L^2$ to $S_1$.
As $Z^m$ vanishes uniformly, we get
\begin{equation}
    \lim_{m \to \infty} \mathbb E \!\left[ \sup_{t \in [0,1]} |S^m_t - S_t|^2 \right]\! = 0.
\end{equation}
Since $S^m_t \sim \mu_t$ for $t \in \{0,2^{-m},\ldots,1\}$, we have by continuity of $t \mapsto \mu_t$ that $S_t \sim \mu_t$ for every $t \in [0,1]$.
By the Lipschitz property of $\sigma$, for any $t \in [0, 1]$, we find that
\begin{equation}
    \lim_{m \to \infty} \mathbb E \!\left[ \int_0^t |\sigma_t(S^m_s) - \sigma_t(S_s)|^2 \, \D s \right]\! = 0.
\end{equation}
Thus, for any $t \in [0, 1]$, the It\^o isometry yields
\begin{equation}
    \int_0^t \sigma_s(S_s) \, \D B_s = \lim_{m \to \infty} \int \sigma_s(S^m_s) \, \D B_s = \lim_{m \to \infty} \tilde S^m_t - X^{n(m)}_0 = S_t - S_0, 
\end{equation}
and so $S$ is the unique strong solution of $dS_t = \sigma_t(S_t) \, \D B_t$, with $S_0 \sim \mu_0$.

Fix $\varepsilon > 0$ and $\hat m \in \N$.
Then there exists $M \ge \hat m$ such that for all $m \ge M$
\begin{equation}
    \mathbb E \!\left[ \sup_{t \in [0,1]} |S^m_t - S_t| \right]\! < \varepsilon.
\end{equation}
Recall that by construction we have $(S^m_0, S^m_{t^{\hat m}_1},\ldots, S^m_1) = (X^{n(m)}_0, X^{n(m)}_{t^{\hat m}_1}, \ldots, X^{n(m)}_1)$ in law, from which we deduce that $(X^{n(m)})_{m \in \N}$ converges in finite dimensional distributions to $S$.

Finally, to see uniqueness of the limit, note that by Lemma \ref{lem:cpct-continuity-points} the sequence $(\Law(X^n))_{n \in \N}$ is precompact with respect to convergence in finite dimensional distributions.
By the first part of the proof any subsequence of $(X^n)_{n \in \N}$ admits a subsequence that converges in finite dimensional distributions to $S$.
Hence, we conclude that $\Law(S)$ is the unique limit by recalling that the finite dimensional distributions separate points on the space of probability measures on the Skorohod space $\mathcal D([0,1];\R^d)$.
\end{proof}

Having established \Cref{prop:fdd_convergence}, we now extend the result from compact subsets to the whole of $\R^d$ in order to complete the proof of \Cref{thm:fdd_convergence}.

\begin{proof}[Proof of Theorem \ref{thm:fdd_convergence}]
    In order to apply Proposition \ref{prop:fdd_convergence}, for any radius $R > 0$, define the diffusion coefficients $(\sigma^{n,R})_{n \in \N}$ by
    \begin{equation}
        \sigma_t^{n,R}(x) \coloneqq
        \begin{cases}
            \sigma_t^n(x) & |x| \le R, \\
            \sigma_t^n\!\left(R \frac{x}{|x|}\right)\! & \text{else}.
        \end{cases}
    \end{equation}
    For each $n \in \N$, let $\tau^{n, R} \coloneqq \inf \{t > 0 \colon \; |X^n_t| \geq R\}$ and define a process $X^{n, R}$ by
    \begin{equation}\label{eq:thm_fdd_convergence_restricted_sigma}
    	X^{n, R}_t \coloneqq X^n_{t \wedge \tau^{n, R}} + \int_{\tau^{n, R}}^{t \wedge \tau^{n, R}}\sigma^{n, R}_t(X^{n,R}_t) \D B^n_t, \quad t \in [0, 1].
    \end{equation}
    Then $X^{n, R}$ is a weak solution of the SDE $\D X^{n, R}_t = \sigma^{n, R}_t(X^{n, R}_t) \D B^n_t$ with $X^{n, R}_0 \sim \mu^n_0$.
    The map $x \mapsto \sigma_t^{n,R}(x)^2$ is both Lipschitz continuous and bounded, uniformly in $(t,n) \in [0,1] \times \N$. By \Cref{it:assumption_3}, $C \coloneqq \sup_{n \in \N} \mathbb E[|X^n_1|^2] < \infty$, and thus by \eqref{eq:thm_fdd_convergence_restricted_sigma} and Doob's martingale inequality,
    \begin{equation}
        \label{eq:thm_fdd_convergence_weak1}
        \mathbb E \!\left[ \sup_{t \in [0,1]} |X^{n,R}_t - X^n_t| \wedge 1 \right]\! \le \mathbb P \!\left( \sup_{t \in [0,1]} |X^{n}_t| \ge R \right)\! \le \frac{C}{R^2}.
    \end{equation}
    Moreover, $\sup_{n \in \N} \E[|X^{n, R}_1|^2] \leq C$.
    
    We see that the conditions of \Cref{prop:fdd_convergence} are fulfilled, and thus there exists a diffusion coefficient $\sigma^{R}$ such that the SDE $dX^R_t = \sigma_t^R(X^R_t) \, \D B_t$ admits a unique strong solution $X^R$ with $X^R_0 \sim \mu_0$, and $(X^{n,R})_{n \in \N}$ converges in finite dimensional distributions to $X^R$.
    Observe that, for $R' \ge R$, the corresponding diffusion coefficients are compatible, in the sense that
    \begin{equation}
        \sigma^{R}_t(x) = \sigma^{R'}_t(x) \quad \mbox{for every } x \in B_R \mbox{ and almost every } t\in [0,1].
    \end{equation}
    Defining $\sigma_t(x) \coloneqq \sum_{R = 0}^\infty \mathds 1_{[R,R+1)}(|x|) \sigma_t^R(x)$, the properties of $\sigma^R$ given by \Cref{prop:fdd_convergence} imply that $\sigma^2$ is locally Lipschitz continuous, uniformly in $t \in [0,1]$, and that, for each $x \in \R^d$, there exist constants $c, C > 0$ such that $c\, \id \leq \sigma_t(x)^2 \leq C \, \id$, for $t \in [0, 1]$. In particular, $\sigma$ is locally Lipschitz continuous, uniformly in $t \in [0, 1]$, by \Cref{rem:square-root}. Skorohod \cite[Chapter 3, Section 3]{Sk65} shows that weak existence and pathwise uniqueness hold for the SDE $\D X_t = \sigma_t(X_t) \D B_t$ with $X_0 \sim \mu_0$. Therefore there exists a unique strong solution $X$, by Yamada and Watanabe \cite[Corollary 1]{YW}.
    Defining $\tau^R \coloneqq \inf \{ t > 0 \colon |X^R_t| \ge R \}$ and $\tau \coloneqq \inf \{ t > 0 \colon |X_t| \geq R\}$, pathwise uniqueness implies that, almost surely, $\tau^R = \tau$ and
    \begin{equation}
        \label{eq:thm_ffd_convergence_XR_X}
        X^R_{t \wedge \tau^R} = X_{t \wedge \tau^R}, \quad \text{for all } t \in [0,1].
    \end{equation}
    Thus, by the continuity of the paths of $X^R$, the convergence of $(X^{n,R})_{n \in \N}$ in finite dimensional distributions to $X^R$, and Doob's martingale inequality,
    \begin{align}
        \mathbb P \!\left( \sup_{t \in [0,1]} |X_t| > R \right)\!
        &= 
        \mathbb P\!\left( \sup_{t \in [0,1]} |X_t^R| > R \right)\!
        \leq \sup_{I \subseteq \mathbb Q \cap [0,1]}
        \mathbb P\!\left( \sup_{t \in I} |X_t^R| > R \right)\!
        \\
        &\le
        \sup_{I \subseteq \mathbb Q \cap [0,1]} \sup_{n \in \N} \mathbb P \!\left( \sup_{t \in I} |X_t^{n,R}| > R \right)\!
        \le \frac{C}{R^2}.
    \end{align}
    Therefore, similarly to \eqref{eq:thm_fdd_convergence_weak1}, we find that
    \begin{equation}
        \label{eq:thm_fdd_convergence_weak2}
        \mathbb E \!\left[ \sup_{t \in [0,1]} |X^R_t - X_t| \wedge 1 \right]\! \le \mathbb P\!\left( \sup_{t \in [0,1]} |X_t| \ge R \right)\! \le \frac{C}{(R - 1)^2}.    
    \end{equation}
    For any finite subset $I \subseteq \mathbb Q \cap [0,1]$, \eqref{eq:thm_fdd_convergence_weak1} and \eqref{eq:thm_fdd_convergence_weak2} imply that
    \begin{align}
        & \lim_{n \to \infty} \mathbb E\!\left[ \sup_{t \in I} |X_t^n - X_t| \wedge 1 \right]\! 
        \\
        & \qquad \qquad \le 
        \liminf_{n \to \infty} \mathbb E\!\left[ \sup_{t \in [0,1]} |X_t^n - X_t^{n,R}| \wedge 1 + \sup_{t \in I} |X_t^{n,R} - X_t^{R}| + \sup_{t \in [0,1]} |X_t^{R} - X_t| \wedge 1 \right]\!
        \\
        & \qquad \qquad \le
        \frac{2C}{(R-1)^2} + \liminf_{n \to \infty} \mathbb E \!\left[ \sup_{t \in I} |X_t^{n,R} - X_t^{R}| \wedge 1\right]\!
        =\frac{2C}{(R-1)^2},
    \end{align}
    using convergence of $(X^{n,R})_{n \in \N}$ in finite dimensional distributions to $X^R$. Taking $R \to \infty$, we conclude that $(X^n)_{n \in \N}$ converges to $X$ in finite dimensional distributions.
\end{proof}

\appendix

\section{The Markov and strong Markov property}\label{app:A}

For the definitions and properties of Markov and strong Markov processes, we refer to \cite{BlGe68,KaSh12}.

\begin{lemma}
    \label{lem:strong_Markov}
    Let $(X_t)_{t \in [0, 1]}$ be a c\`adl\`ag process on a filtered probability space $(\Omega, \F, \P, (\F_t)_{t\in [0,1]})$ under the usual conditions.
    Suppose that $X$ satisfies the Markov property for all times $t \in [0, 1]$ and satisfies the strong Markov property for all $\varepsilon > 0$ and all finite stopping times $\tau \ge \varepsilon > 0$. Then $X$ is a strong Markov process.
\end{lemma}

\begin{proof}
    Let $g \colon \R^d \to \R$ be measurable and bounded, $t \in \R_+$, and fix $\varepsilon > 0$.
    For a given finite stopping time $\tau$, 
    we consider the event $A_\varepsilon \coloneqq \{ \tau \le \varepsilon \} \in \F_\varepsilon$, and define the finite stopping time
    \begin{equation}
        \tau_\varepsilon \coloneqq \varepsilon \mathbbm 1_{A_\varepsilon} + \tau \mathbbm 1_{A_\varepsilon^\textrm{c}} \geq \varepsilon.
    \end{equation}
    In the following, all equalities should be understood $\P$-almost surely. On the event $A_\varepsilon^\textrm{c}$, we have by assumption that
    \begin{equation}
        \E[g(X_{\tau + t}) \mid \F_{\tau}] = \E[ g(X_{\tau_\varepsilon + t}) \mid \F_{\tau_\varepsilon}] = \E[g(X_{\tau_\varepsilon + t}) \mid \sigma(X_{\tau_\varepsilon})] = \E[g(X_{\tau + t}) \mid \sigma(X_{\tau})].
    \end{equation}
    In particular, since $A^\textrm{c}_\varepsilon \nearrow \{ \tau > 0 \}$ as $\varepsilon \searrow 0$, we find that
    	$\E[ g(X_{\tau + t}) \mid \F_{\tau}] = \E[ g(X_{\tau + t}) \mid \sigma(X_{\tau})]$ on $\{ \tau > 0 \}$.
    On the other hand, we have by the Markov property of $X$ that $\E[g(X_t) \mid \F_0] = \E[g(X_t) \mid \sigma(X_0)]$.
    Combining these two observations yields that, for all finite stopping times $\tau$,
    \begin{align}
        \E[g(X_\tau) \mid \F_\tau] & = \mathbbm 1_{\{ \tau = 0 \}} \E[g(X_t) \mid \F_0] + \mathbbm 1_{\{ \tau > 0 \}} \E[g(X_{\tau + t}) \mid \F_\tau]\\
        & =  \mathbbm 1_{\{ \tau = 0 \}} \E[g(X_t) \mid \sigma(X_0)] + \mathbbm 1_{\{ \tau > 0 \}} \E[g(X_{\tau + t}) \mid \sigma(X_\tau)]\\
        & = \E[ g(X_{\tau + t}) \mid \sigma(X_{\tau})].      
    \end{align}
    Hence $X$ is a strong Markov process.
\end{proof}

\begin{lemma}\label{lem:markov-diffusion}
	Let $B$ be a standard Brownian motion on $\R^d$ and $\tilde \sigma$ an $\F^B$-predictable process taking values in the set of positive semidefinite matrices. Let $X$ be a square-integrable It\^o process satisfying $\D X_t = \tilde \sigma_t \D B_t$ for $t \in [0, 1]$. If $X$ is a Markov process, then there exists a measurable function $\sigma \colon [0, 1] \times \R^d \to \R^{d \times d}$ taking values in the set of positive semidefinite matrices such that
	\begin{equation}\label{eq:markov-ito}
		\D X_t = \sigma_t(X_t) \D B_t, \quad t \in [0, 1].
	\end{equation}
\end{lemma}

\begin{proof}
	Fix $s \in [0, 1]$. Then, for any $t \in [s, 1]$,
	\begin{equation}
		\E [|X_t|^2 - |X_s|^2 \mid \F^X_s] = \E [|X_t|^2 - |X_s|^2 \mid \sigma(X_s)],
	\end{equation}
	almost surely, by the Markov property. Thus
	\begin{equation}
		\begin{split}
			\tilde \sigma_s^2 & = \lim_{t \searrow s}\frac{1}{t -s}\E\!\left[\int_s^t \tilde \sigma_r^2 \D r \mid \F^X_s\right]\! = \lim_{t \searrow s}\frac{1}{t -s}\E\!\left[\int_s^t \tilde \sigma_r^2 \D r \mid \sigma(X_s)\right]\!,
		\end{split}
	\end{equation}
	almost surely. Hence $\tilde \sigma_s$ is almost surely adapted to the sigma-algebra generated by $X_s$ and the representation \eqref{eq:markov-ito} holds.
\end{proof}

\section{Auxiliary results for the examples of Section \ref{sec:necessity}}\label{app:B}

\begin{lemma}\label{lem:app-rem-1}
	In the setting of \Cref{rem:ex-cadlag}, for each $t_0 \in (0, 1]$, the process $X^{t_0}$ mimics $\mu$ on $[t_0, 1]$. Moreover, as $t_0 \to 0$, there exists a limit $X$ of $X^{t_0}$ in distribution such that $(X_t)_{t \in [t_1,1]} \sim (X^{t_1}_t)_{t \in [t_1,1]}$, for any $t_1 \in (0, 1]$.
\end{lemma}

\begin{proof}
	For $t_0 \in (0, 1)$ and $n \in \N$, denote $A_{n,t_0} \coloneqq \{ \sum_{k = 1}^n \xi_k < \Lambda_t - \Lambda_{t_0} \le \sum_{k = 1}^{n + 1} \xi_k \}$. Note that the sum of $n$ independent exponential distributions with rate parameter $1$ is distributed according to a Gamma distribution with shape paramter $n$ and rate parameter $1$.
        Therefore we have
        \begin{equation}
            \label{eq:rem.probs}
            \mathbb P[A_{n,t_0}] = \int_0^{\Lambda_{t} - \Lambda_{t_0}} \frac{x^{n-1}}{(n-1)!} e^{-x + x + \Lambda_{t} - \Lambda_{t_0}} \, \D x
            =
            \frac{(\Lambda_t - \Lambda_{t_0})^n}{n!} e^{-(\Lambda_t - \Lambda_{t_0})}.
        \end{equation}
	For $i = 1, 2$, we compute
    \begin{align}
        \mathbb P[ X_t^{t_0} \in S_t^i \mid X^{t_0}_{t_0} \in S_{t_0}^2] 
        &= \sum_{n \in \N \cup \{ 0 \} }\mathbbm 1_{2\mathbb Z + i }(n) \P \!\left[ A_{n, t_0}\right]\!
        \\
        &= \sum_{n \in \N \cup \{0 \} }\mathbbm 1_{2 \mathbb Z + i}(n) \frac{(\Lambda_t - \Lambda_{t_0})^n}{n!} e^{-(\Lambda_t - \Lambda_{t_0})}
        \\
        &=
        \begin{cases}
            \frac{\sinh(\Lambda_t - \Lambda_{t_0})}{e^{(\Lambda_t - \Lambda_{t_0})}} & i = 1,
            \\
            \frac{\cosh(\Lambda_t - \Lambda_{t_0})}{e^{(\Lambda_t - \Lambda_{t_0})}} & i = 2.
        \end{cases}
    \end{align}
    Therefore, by rotational symmetry, we have that $X^{t_0}_{t_0} \sim \mu_{t_0}$ implies that $X_t^{t_0} \sim \mu_t$ for all $t \in [t_0,1]$.
    Now note that by the memoryless property of the exponential distribution we have, for $0 < t_0 < t_1 \le 1$,
    \begin{equation}
        (X_t^{t_0})_{t \in [t_1,1]} \sim (X^{t_1}_t)_{t \in [t_1,1]}.
    \end{equation}
    Hence, as $t_0 \searrow 0$, there exists a limiting process $X = (X_t)_{t \in [0,1]}$ with c\`adl\`ag paths that satisfies the desired properties.
\end{proof}

\begin{lemma}\label{lem:app-rem-2}
	Let $X$ be as in \Cref{rem:ex-cadlag} and let $f \colon \R^d \to [0,1]$ be measurable. Then, for $t \in [0,1]$,
    \begin{equation}\label{eq:rem-claim}
        \E[f(X_{t}) \mid \F^X_0] = \lim_{t_0 \searrow 0} \E[f(X_{t}) \mid \mathcal F^X_{t_0} ] = \E[f(X_{t})]\quad \mbox{a.s.}
    \end{equation}
\end{lemma}

\begin{proof}
	Note that the first equality in \eqref{eq:rem-claim} is due the martingale convergence theorem.
    By the Markov property at $t_0$, we have that, for $0 < t_0 < t$,
    \begin{equation}
        \E[f(X_{t}) \mid \mathcal F^X_{t_0} ] = \E[f(X_{t})\mid X_{t_0} ].
    \end{equation}
    By construction, given the starting point $X_{t_0} = x \in S_{t_0}^2$, we have $X_{t} \sim X_{t}^{t_0}$.
    Now use the independence of $(U_n)_{n \in \N}$ and $(\xi_k)_{k \in \N}$ to compute that
    \begin{align}
        \mathbb E[f(X_t) \mid X_{t_0} = x] &= \sum_{n \in \N \cup \{ 0 \}} \mathbb E[ \mathbbm 1_{A_{n,t_0}} ( \mathbbm 1_{2\mathbb Z}(n) f(V_n) + \mathbbm 1_{2\mathbb Z + 1}(n) f(U_n) ]
        \\
        &=\sum_{n \in \N \cup \{ 0 \} } \P[ A_{n,t_0} ] \!\left( \mathbbm 1_{2\mathbb Z}(n)\mathbb E[f(V)] + \mathbbm 1_{2\mathbb Z + 1}(n)\mathbb E[f(U)]\right)\!,
    \end{align}
    where $U \sim \text{Unif}(S^1_1)$ and $V \sim \text{Unif}(S^2_1)$.
    Since the law of $X_t$ is given explicitly by $\mu_t = \frac12 (L((0,0,\sqrt{t}U)) + L(\sqrt{t}U,0,0))$, it remains to prove that
    \begin{equation}
        \sum_{n \in \N \cup \{ 0 \}} \mathbb P[A_{2n,t_0}] \to \frac12\quad\mbox{as } t_0\searrow 0.
    \end{equation}
    Using \eqref{eq:rem.probs} and noting that $\Lambda_t - \Lambda_{t_0}$ diverges to $+\infty$ for $t_0 \searrow 0$, we find that
    \begin{equation}
        \sum_{n \in \N \cup \{ 0 \}} \mathbb P[A_{2n,t_0}] = \frac{\cosh(\Lambda_t - \Lambda_{t_0})}{e^{\Lambda_t - \Lambda_{t_0}}} \to \frac12,
    \end{equation}
    and conclude that \eqref{eq:rem-claim} holds.
\end{proof}

\begin{lemma}\label{lem:app-reg}
	Let $n \in \N$ and set $t_0 = 2^{-(n+1)}$, $t_1 = 2^{-n}$. Then, in the setting of \Cref{prop:ex-regularized}, with $\Sccc^i_{t} \coloneqq \{\exists x \in S^i_{a_i(t)} \colon \; |x - Y_t| < t \}$ for $t \in [0, 1]$, $i = 1,2$, there exists a constant $C > 0$ such that
	   \begin{equation}
    	    \P[\Sccc^i_{t_0} \bigtriangleup \Sccc^i_{t_1}] \le C t_0.
	   \end{equation}
\end{lemma}

\begin{proof}
	Fix $t_0 = 2^{-(n+1)}$, $t_1 = 2^{-n}$ for some even integer $n \in \N \cup \{ 0 \}$. Then we have that $[2^{-(n+1)}, 2^{-n}] \in I_1$, and so $a_2(t)$ and $\mu_t |_{\X_2}$ are constant on this interval.    
    Subsequently, we will repeatedly employ the estimates
    \begin{equation}\label{eq:ex_regularized.estimates}
    \begin{split}
        & \E[|X_t|^4 + |N_{t^{14}}|^4] \le \frac{a_1(t)^2 + a_2(t)^2}{2} + 3 t^{28} \le C t^2,\\
        & \P[|N_{t^{14}}| \ge a] \le \frac{t^{14}}{a^2}, \quad a > 0.
	\end{split}
    \end{equation}

	Similarly as in \Cref{prop:ex-discontinuous}, we split the vector $Y$ into two parts where $Y^1_t$ denotes the first two coordinates and $Y^2_t$ the last two components of $Y_t$, for $t \in [0, 1]$. We aim to bound the probability of $Y^i$ leaving a small ball around $Y^i_{t_0}$ in the time interval $[2^{-(n +1)}, 2^{-n}]$.
	Also define the events $\Scc^i_{t} \coloneqq \{\exists x \in S^i_{a_i(t)} \colon \; |x - Y_t| < t^2\}$ for $t \in [0, 1]$, $i = 1,2$.
	
	First we compute the upper bound
    \begin{align}\label{eq:ex_regularized.ub}
        \E[|Y^i_t|^2] \le \E[|X_t^i|^2 + |N_{t^{14}}|^2] \le \frac12 a_i(t) + t^{14}.
    \end{align}
    In the case that $i = 2$ we get the lower bound
    \begin{align}
        \E[|Y_{t_0}^2|^2 \mathbbm 1_{\Scc_{t_0}^2}] & \ge  \E[|X_{t_0} + N_{t_0^{14}}|^2 \mathbbm 1_{\Scc_{t_0}^2}; |N_{t_0^{14}}| < t_0^2]\\
        &\ge
        \E[|X_{t_0} + N_{t_0^{14}} |^2\mathbbm 1_{\Scc_{t_0}^2}] - \E[|X_{t_0} + N_{t_0^{14}}|^2; |N_{t_0^{14}}| \ge t_0^2]
        \\
        &\ge \E[|X_{t_0}|^2\mathbbm 1_{\Scc_{t_0}^2}] - \E[|X_{t_0}|^2 + |N_{t_0^{14}}|^2; |N_{t_0^{14}}| \ge t_0^2],
    \end{align}
    using the independence of $X_{t_0}$ and $N_{t_0^{14}}$. We use this independence again, together with the Cauchy-Schwarz inequality and the estimates \eqref{eq:ex_regularized.estimates}, to show that
    
    \begin{align}
    	 \E[|Y_{t_0}^2|^2 \mathbbm 1_{\Scc_{t_0}^2}] &\ge \frac12 a_2(t_0) - \E[|X_{t_0}|^2 + |N_{t_0^{14}}|^2; |N_{t_0^{14}}| \ge t_0^2 ] \\
        &\ge \frac12 a_2(t_0) - \!\left(\P[|N_{t_0^{14}}| \ge t_0^2] \E[|X_{t_0}|^4 + |N_{t_0^{14}}|^4]\right)\!^\frac12
        \\
        &\ge \frac12 a_2(t_0) - C t_0^6.
    \end{align}
    
    Recalling that $a_2$ is constant on $[t_0,t_1]$, \eqref{eq:ex_regularized.ub} implies that $\E[|Y_{t_1}|^2] \leq \frac12 a_2(t_0) + t_1^{14}$.
    Since $Y$ is a martingale we have
    \begin{equation}
    	 \E[|Y_{t_1}^2 - Y^2_{t_0}|^2] = \E[|Y_{t_1}^2|^2 - |Y_{t_0}^2|^2] \le \E[|Y_{t_1}^2|^2] - \E[|Y_{t_0}^2|^2\mathbbm{1}_{\Scc_{t_0}^2}].
    \end{equation}
    Combining the upper and lower bounds above gives us
    \begin{equation}
        \label{eq:ex_regularized.claim1}
        \E[|Y_{t_1}^2 - Y^2_{t_0}|^2] \le \frac12 a_2(t_0) + t_1^{14} - \!\left(\frac12 a_2(t_0) - Ct_0^6\right)\! \le C_0 t_0^6.
    \end{equation}
    for some $C_0 > 0$.
    By Doob's maximal inequality we obtain
    \begin{equation}
        \label{eq:ex_regularized.2_Doob1}
        \P[\sup_{t \in [t_0,t_1]} |Y^2_t - Y_{t_0}^2| \ge t_0 / 2] \le 4 C_0 t^4_0.
    \end{equation}

    Next, we prove a similar bound to \eqref{eq:ex_regularized.claim1} when $i = 1$. We claim that, for some constant $C_1 > 0$,
    \begin{equation}
        \label{eq:ex_regularized.claim2}
        \E[|Y_{t_1}^1 - Y_{t_0}^1|^2 ; (\Scc^1_{t_0})^\textrm{c}] \le C_1 t_0^3.
    \end{equation}
    Since $Y$ is a martingale, we have
    \begin{equation}
	    \begin{split}
	    	\E[|Y_{t_1}^1 - Y_{t_0}^1|^2 ; (\Scc^1_{t_0})^\textrm{c}] & = \E[|Y^1_{t_1}|^2 - |Y^1_{t_0}|^2;(\Scc^1_{t_0})^\textrm{c}] \leq \E[|Y^1_{t_1}|^2 ;(\Scc^1_{t_0})^\textrm{c}]\\
	    	& \leq \E[|Y^1_{t_1}|^2 ;(\Scc^1_{t_0} \cup \Scc^2_{t_0})^\textrm{c}] + \E[|Y^1_{t_1}|^2 ;\Scc^2_{t_0}]
	    \end{split}
    \end{equation}
    We estimate each term separately, using the Cauchy-Schwarz inequality and the estimates \eqref{eq:ex_regularized.estimates} in each case. We first bound
    \begin{align}
        \E[|Y_{t_1}^1|^2 ; (\Scc^1_{t_0} \cup \Scc^2_{t_0})^\textrm{c}]
        \le
        \E[|Y_{t_1}|^4]^\frac12 \P[ (\Scc^1_{t_0} \cup \Scc^2_{t_0})^\textrm{c} ]^\frac{1}{2} 
        \le C t_1 \P[ |N_{t^{14}_0}| \ge t_0^2]^\frac12 \le C_2 t_0^{6}.
    \end{align}
    To bound the second term, we additionally apply \eqref{eq:ex_regularized.2_Doob1} to get
    \begin{align}
        \E[|Y_{t_1}^1|^2 ;\Scc^2_{t_0} ]
        &=
        \E[|Y_{t_1}^1|^2 ;\Scc^2_{t_0}, \; |Y^2_{t_1} - Y^2_{t_0}| \le t_0 / 2 ] +
        \E[|Y_{t_1}^1|^2 ;|Y^2_{t_1} - Y^2_{t_0}| \ge t_0 / 2]
        \\
        &\le
        \E[|Y_{t_1}^1|^2 ;(\Scc^1_{t_1})^\textrm{c}] + \!\left( \E[ |Y_{t_1}|^4 ] \P[ |Y_{t_1}^2 - Y_{t_0}^2| \ge t_0 / 2] \right)\!^\frac12
        \\
        &\le
        C_2 t_0^6 + \!\left( C t_0^2 \cdot 4 C_0 t_0^4 \right)\!^\frac12 \le C_3 t_0^3.
    \end{align}   
    Combining the two preceding inequalities yields \eqref{eq:ex_regularized.claim2}.
   	As before, Doob's maximal inequality implies
    \begin{equation}
        \label{eq:ex_regularized.2_Doob2}
        \P[\sup_{t \in [t_0,t_1]} |Y_t^1 - Y_{t_0}^1| \ge t_0 / 4; \; (\Scc^1_{t_0})^\textrm{c}] \le 16 C_1 t_0.    
    \end{equation}
    
    Note that, for $i = 1, 2$,
    \begin{equation}
    	\begin{split}
    		\P [ \Sccc^i_{t_0} \bigtriangleup \Sccc^i_{t_0}] & \leq \P [\Scc^i_{t_0} \bigtriangleup \Sccc^i_{t_1}] + \P[\Sccc^i_{t_0} \setminus (\Scc^i_{t_0} \cup \Sccc^i_{t_1})],
    	\end{split}
    \end{equation}
    and
    \begin{equation}
    	\P[\Sccc^i_{t_0} \setminus (\Scc^i_{t_0} \cup \Sccc^i_{t_1})] = \P[t_0^2 \leq |N_{t_0^{14}}| < t_0, \; |N_{t_1^{14}}| \geq t_1] \leq t_0^{10}.
    \end{equation}
    Hence, to prove the conclusion of the lemma, we only require bounds on $\P [\Scc^i_{t_0} \bigtriangleup \Sccc^i_{t_1}]$, for $i = 1, 2$.
   
   Now consider the events
    \begin{equation}
        A_{t_0} \coloneqq \{ |Y_{t_1}^2 - Y_{t_0}^2| \le t_0 / 4 \} \quad
        \mbox{and}\quad
        B_{t_0} \coloneqq \Scc^2_{t_0} \cap \{ |Y_{t_1}^1 - Y_{t_0}^1| \le t_0 / 4 \},
    \end{equation}
    and observe that \eqref{eq:ex_regularized.2_Doob1} and \eqref{eq:ex_regularized.2_Doob2} imply that
    \begin{equation}
        \P[A_{t_0}^\textrm{c}] \le 16C_0 t_0^4 \quad \mbox{and} \quad
        \P[\Scc^2_{t_0} \setminus B_{t_0}] \le 16 C_1 t_0.
    \end{equation}
   We calculate that
   \begin{align}
        \P[\Scc_{t_0}^1 \bigtriangleup \Sccc_{t_1}^1] & = \P[\Scc_{t_0}^1 \setminus (\Scc_{t_0}^1 \cap \Sccc_{t_1}^1)] + \P[\Sccc_{t_1}^1 \setminus (\Scc_{t_0}^1 \cap \Sccc_{t_1}^1)] \\
        & = \P[\Scc^1_{t_0}] + \P[\Sccc^1_{t_1}] - 2\P[\Scc^1_{t_0} \cap \Sccc^1_{t_1}]\\
        & \le 2(\P[\Sccc_{t_1}^1] - \P[\Scc_{t_0}^1 \cap \Sccc_{t_1}^1]) + t_0^{10},
   \end{align}
   since $\P[\Sccc^1_{t_1}] \geq \P[\Scc^1_{t_0}, |N_{t_0^{14}}| < t_0^2] \geq \P[\Scc^1_{t_0}] - t_0^{10}$.
   Note that we have the inclusions
   \begin{equation}
   		(\Scc_{t_0}^1)^\textrm{c} \cap \Sccc_{t_1}^1 \subseteq (\Scc^1_{t_0} \cup \Scc^2_{t_0})^\textrm{c} \cup (\Scc_{t_0}^2 \cap \Sccc_{t_1}^1), \quad \Scc^2_{t_0} \cap \Sccc^1_{t_1} \subseteq A_{t_0}^\textrm{c}.
   \end{equation}
   Hence
   \begin{align}
        \P[\Scc_{t_0}^1 \bigtriangleup \Sccc_{t_1}^1] \le 2\P[(\Scc^1_{t_0})^\textrm{c} \cap \Sccc^1_{t_1}] + t_0^{10}
        & \le 2 \P[(\Scc_{t_0}^1 \cup \Scc_{t_0}^2)^\textrm{c} \cup (\Scc_{t_0}^2 \cap \Sccc_{t_1}^1)] + t_0^{10}
        \\
        & \le 2 \P[|N_{t^{14}}| \ge t_0^2] + 2 \P[A_{t_0}^\textrm{c}] + t_0\\
        & \le 3 t_0^{10} + 32 C_0 t_0^4.
    \end{align}
    On the other hand, we have
    \begin{align}
        \P[\Scc_{t_0}^2 \bigtriangleup \Sccc_{t_1}^2] \le \P[\Scc^2_{t_0} \setminus B_{t_0}] + \P[|N_{t_0^{14}}| \geq t_0]\le C t_0.
    \end{align}
    Thus we have shown the claim in the case that $n$ is even. Due to symmetry of the arguments we also obtain the desired result when $n$ is odd.	
\end{proof}

\bibliographystyle{abbrvnat}
\bibliography{joint_biblio}
\end{document}